\documentclass[11pt]{article}
\usepackage{amsfonts}
\usepackage{amsmath}
\usepackage{amssymb}
\usepackage{epsfig}
\usepackage{enumerate}
\usepackage{citesort}

\usepackage[show]{ed}
\usepackage{color}
\usepackage[normalem]{ulem}

\title{Regularized combined field integral equations for acoustic transmission problems}
\author{Yassine Boubendir, V\'{\i}ctor Dom\'{\i}nguez, David Levadoux, 
Catalin Turc\\ \small 
New Jersey Institute of Technology, Universidad P\'ublica de Navarra, Spain,\\ \small
 ONERA France, New Jersey Institute of Technology\\ \small
boubendi@njit.edu,\ victor.dominguez@unavarra.es,\ david.levadoux@onera.fr,
catalin.c.turc@njit.edu}

\newtheorem{theorem}{Theorem}[section]
\newtheorem{lemma}[theorem]{Lemma}

\newtheorem{remark}[theorem]{Remark}
\newenvironment{proof}{\hspace{0.5cm} {\bf Proof.}}
{$\quad {}_\blacksquare$\vspace{0.3cm}}

\date{}

\begin{document}
\maketitle
\begin{abstract}
  We present a new class of well conditioned integral equations for the solution 
of
 two and three dimensional scattering problems by homogeneous penetrable 
scatterers. Our novel boundary integral equations result from suitable 
representations of the fields inside and outside the scatterer as 
combinations of single and double layer potentials acting on suitably defined 
regularizing operators. The regularizing operators are constructed to be 
suitable approximations of the admittance operators that map the transmission 
boundary conditions to the exterior and respectively interior Cauchy data on the 
interface between the media. The latter operators can be expressed in terms of 
Dirichlet-to-Neumann operators. We refer to these regularized boundary integral 
equations as Generalized Combined Source Integral 
Equations (GCSIE). The ensuing GCSIE are shown to be integral equations of 
the second kind in the case when the interface of material discontinuity is a 
smooth curve in two dimensions and a smooth surface in three dimensions.   
\end{abstract}

\section{Introduction\label{intro}}

\parskip 2pt plus2pt minus1pt

Numerical methods based on integral equation formulations for the solution of 
scattering problems, when applicable, have certain advantages over those that 
use volumetric formulations, largely owing to  the dimensional reduction and 
the 
explicit enforcement of the radiation conditions. A crucial requirement of 
reformulating a linear,
constant-coefficient PDE in terms of boundary integral equation formulations is 
that the latter are well-posed. In the case when the boundary of the scatterer 
is regular enough, there is a myriad of possibilities to derive well-posed 
boundary integral equations. Amongst those, the most widely used methodology of 
deriving well-posed integral equations for solution of scattering problems 
relies on Combined Field Integral Equations 
(CFIE)~\cite{BrackhageWerner,BurtonMiller,muller,KressRoach}. Here in what 
follows we assume that the boundary of the scatterer is regular enough; the 
case of less regular interfaces (e.g. Lipschitz) is much less understood. In 
the scalar case, this methodology seeks scattered fields in terms of suitable 
linear combinations of single and double layer potentials so that the 
enforcement of the boundary conditions leads to boundary integral equations 
(CFIE) whose underlying operators are 
Fredholm in suitable boundary trace spaces of scattering problems. The 
well-posedness of CFIE is then 
settled via uniqueness arguments for the Helmholtz equation with certain 
boundary conditions (e.g. impedance boundary conditions, transmission boundary 
conditions).   

Solvers based on integral 
equation formulations of scattering problems lead to systems of linear 
equations 
that involve dense matrices that can be very large in the high-frequency 
regime. 
Due to the large size of the underlying matrices, the solution of these linear 
algebra problems relies on Krylov subspace iterative solvers and is greatly 
facilitated by the availability of fast algorithms to 
perform matrix vector 
products. Thus, the efficiency of scattering solvers based on integral 
equations 
hinges a great deal on the spectral properties of the integral operators that 
enter the integral formulations, spectral properties themselves that influence 
the speed of convergence of iterative solvers. Thus, boundary integral 
equations of the second kind (e.g. equations whose operators are compact 
perturbations of identity in appropriate functional spaces) are preferable
for 
the solution of scattering problems. The combined field strategy delivers 
boundary integral equations of the second kind for scalar scattering problems 
with Dirichlet boundary conditions~\cite{BrackhageWerner} and transmission 
boundary conditions~\cite{muller,KressRoach,KittapaKleinman,rokhlin-dielectric}. 
However, it is typically the case that CFIE formulations, although extremely 
reliable, do not necessarily possess the best spectral properties amongst all 
well-posed formulations possible. For instance, in the case of scattering problems with other boundary conditions such as Neumann and impedance 
boundary conditions, various regularization procedures can deliver integral 
equations of the second 
kind~\cite{BrackhageWerner,BurtonMiller,AntoineX,Antoine,Levadoux,br-turc,turc1} with superior spectral properties. The regularization procedure is a means to 
derive {\em systematically} well-conditioned boundary integral formulations of 
scattering problems. The techniques are quite general and can be applied, in 
principle, to any boundary value problem for linear, constant-coefficient PDEs 
or systems of PDEs. 

All the regularization procedures for scalar scattering problems rely on the 
same main steps: (1) the fields (i.e. the solutions of the Helmholtz equations) 
are represented in each domain of interest with the aid of Green's formulas in 
terms of the Cauchy data (i.e. Dirichlet and Neumann traces) on the boundary of 
each domain; (2) an abstract operator $\mathcal{R}$ that maps the given 
boundary conditions to all the Cauchy data needed (e.g: in the case of Neumann 
boundary conditions, the operator $\mathcal{R}$ maps Neumann traces to 
Dirichlet traces and is thus a Neumann-to-Dirichlet operator)---the operator 
$\mathcal{R}$ is defined in terms of Dirichlet-to-Neumann operators and possibly 
their inverses; (3) the fields are represented using Green's formulas where the 
Cauchy data is represented as an operator $\widetilde{\mathcal{R}}$ that 
approximates the operator $\mathcal{R}$ acting on some unknown 
sources/densities. The enforcement of the boundary conditions on this 
representation leads to Generalized 
Combined Source Integral Equations (GCSIE) / Regularized Combined Field 
Integral Equations (CFIER). We note at this stage that by construction, if the 
operator $\mathcal{R}$ were used instead of $\widetilde{\mathcal{R}}$ in the 
GCSIE, then these would consist of the identity operator. In the next step (4) 
the degree of approximation of the operator $\widetilde{\mathcal{R}}$ (that is 
the degree of smoothing of the difference operator 
$\widetilde{\mathcal{R}}-\mathcal{R}$) is established so that second kind 
Fredholm GCSIE are obtained in appropriate boundary trace spaces of scalar 
scattering problems. The desired degree of smoothing is achieved provided that 
(5) the operator $\widetilde{\mathcal{R}}$ is constructed via suitable 
approximations of Dirichlet-to-Neumann operators. In order to meet the 
additional requirement that the ensuing GCSIE operators are injective (and thus 
invertible with continuous inverses), the aforementioned approximations of 
Dirichlet-to-Neumann operators are constructed through {\em 
complexification} of boundary integral operators~\cite{br-turc,turc1}, through 
{\em complexification} of the wavenumbers in the definition of 
Dirichlet-to-Neumann operators for simple geometries such as the 
half-planes/half-spaces corresponding to tangent lines/planes to the boundary 
of the scatterer~\cite{AntoineX,Antoine}---in these cases the 
Dirichlet-to-Neumann maps can be defined by Fourier multipliers, or through 
boundary integral operators corresponding to the wavenumbers of the Helmholtz 
equations under consideration and quadratic partitions of 
unities~\cite{Levadoux,Levadoux1,Levadoux2}. Calder\'on's identities are a 
crucial ingredient in the calculus in part (5). 

We present in this paper novel integral equation formulations of two and three 
dimensional scalar transmission scattering problems that apply the five-step 
program outlined above. These integral equation formulations are actually $2\times 2$ systems of integral equations whose unknowns are certain densities defined on the interface of material discontinuity. There are two possibilities in terms of the functional spaces in which we seek those densities and in which we aim to construct GCSIE operators that are Fredholm of the second kind: (i) we assume that both densities belong to the same boundary Sobolev space and (ii) we assume that one of the densities has one more order of regularity than the other. We note that numerical methods based on GCSIE with property (i) are more accurate, and more amenable to an error analysis. We note that this distinction plays an important role in the construction of appropriate approximations of the operator $\mathcal{R}$ in step (2) which maps the 
difference of exterior and interior Dirichlet and Neumann traces on the 
interface of material discontinuity to the Cauchy data of transmission 
problems. The matrix operator $\mathcal{R}$ is expressed in terms of 
compositions of exterior and interior Dirichlet-to-Neumann operators 
corresponding to two different domains and wavenumbers and inverses of operators 
that involve linear combinations of those. Depending on the case (i) and (ii), the degree of smoothing we require on the difference operator $\widetilde{\mathcal{R}}-\mathcal{R}$ is different. The approximating operators $\widetilde{\mathcal{R}}$, in turn, are constructed per step (5) above via approximations of {\em both} exterior and interior Dirichlet-to-Neumann operators. The degree of these latter approximations is different according to case (i) and case (ii) respectively, and the dimension of the ambient space in which we solve the transmission problem. We rely on Calder\'on's calculus to construct approximations of the Dirichlet-to-Neumann operators in terms of normal derivative of double layer operators corresponding to {\em complex} wavenumbers in case (i) and two dimensional ambient space and in case (ii) and three dimensional ambient space; or linear combinations of those with compositions of the normal derivative of double layer operators and double layer potentials corresponding to {\em imaginary} wavenumbers in case (ii) and three dimensional ambient space. The positivity of the imaginary parts 
of the former operators allows us to establish the injectivity and thus the 
invertibility of the GCSIE operators in both case (i) and case (ii). As it was illustrated in~\cite{turc2}, 
solvers based on the GCSIE formulations, on account of the superior spectral 
properties of these formulations, outperform solvers based on classical 
integral formulations of transmission 
problems~\cite{KressRoach,KittapaKleinman,rokhlin-dielectric}. 

The paper is organized as follows: in Section~\ref{cfie} we present the 
acoustic transmission problems and review the basic properties of scattering 
boundary integral operators; in Section~\ref{cfier} we define and compute the 
admittance operator $\mathcal{R}$ of transmission problems and we set up 
regularized integral equations in the form of Generalized Combined Source 
Integral Equations (GCSIE) that are based on regularizing operators 
$\widetilde{\mathcal{R}}$ that approximate the operators $\mathcal{R}$; in 
Section~\ref{app} we derive sufficient conditions on the regularizing operators 
$\widetilde{\mathcal{R}}$ so that they lead to Fredholm second kind GCSIE; in 
Section~\ref{appY} we construct approximations of Dirichlet-to-Neumann operators 
for each medium that lead via Cald\'eron's calculus to constructions of 
regularizing operators $\widetilde{\mathcal{R}}$; in 
Sections~\ref{case1},~\ref{case2}, and~\ref{case3} we establish the Fredholm 
properties of the GCSIE for various choices of regularizing 
operators $\widetilde{\mathcal{R}}$ in two and three dimensions; finally, in 
Section~\ref{uniq} we establish the well-posedness of the GCSIE.

\section{Integral Equations acoustic transmission\label{cfie}}

\subsection{Acoustic transmission problem}
We consider the problem of evaluating the time-harmonic fields $u^1$ and $u^2$ that result as an incident field $u^{inc}$ impinges upon the boundary
$\Gamma$ of a homogeneous penetrable scatterer $D_2$ which occupies a bounded region in $\mathbb{R}^d,\ d=2,3$. The frequency domain acoustic transmission problem is formulated in terms of finding fields $u^1$ and $u^2$ that are solutions to the Helmholtz equations
\begin{equation}
  \label{eq:Ac_i}
  \Delta u^2+k_2^2 u^2=0 \qquad \mathrm{in}\ D_2,
\end{equation}
\begin{equation}
  \label{eq:Ac_e}
  \Delta u^1+k_1^2 u^1=0\qquad \mathrm{in}\ D_1=\mathbb{R}^d\setminus D_2,
\end{equation}
given an incident field $u^{inc}$ that satisfies
\begin{equation}
  \label{eq:Maxwell_inc}
  \Delta u^{inc}+k_1^2 u^{inc}=0 \qquad \mathrm{in}\ D_1,
\end{equation}
where the wavenumbers $k_i,i=1,2$ are the wavenumbers corresponding to the domains $D_i,i=1,2$ respectively. In addition, the fields $u^{1}$, $u^{inc}$, and $u^2$ are related on the boundary $\Gamma$ by the the following boundary conditions
\begin{eqnarray}
\label{eq:bc}
\gamma_D^1 u^1 + \gamma_D^1 u^{inc} &=&\gamma_D^2 u^2\qquad \rm{on}\ \Gamma \nonumber\\
\gamma_N^1 u^1 + \gamma_N^1 u^{inc}&=&\nu \gamma_N^2u^2\qquad \rm{on}\ \Gamma.
\end{eqnarray}
In equations~\eqref{eq:bc} and what follows $\gamma_D^i,i=1,2$ denote exterior and respectively interior Dirichlet traces, whereas $\gamma_N^i,i=1,2$ denote exterior and respectively interior Neumann traces taken with respect to the exterior unit normal on $\Gamma$. We assume in what follows that the wavenumbers $k_i,i=1,2$ are positive and that the density ratio $\nu$ is also positive.  We note that in the case $d=2$, equations~\eqref{eq:Ac_i}-\eqref{eq:Maxwell_inc} can also model electromagnetic scattering by two-dimensional penetrable obstacles $D_1$, in which case $\nu=1$ (TE case) or $\nu=k_1^2/k_2^2$ (TM case). We assume in what follows that the boundary $\Gamma$ is a closed and smooth curve in $\mathbb{R}^2$ and a closed and smooth surface in $\mathbb{R}^3$. We furthermore require that $u^1$ satisfies Sommerfeld radiation conditions at infinity:
\begin{equation}\label{eq:radiation}
\lim_{|r|\to\infty}r^{(d-1)/2}(\partial u^1/\partial r - ik_1u^1)=0.
\end{equation}
Under the assumption that $k_1$, $k_2$, and $\nu$ are real and positive, it is well known that the systems of partial differential equations~\eqref{eq:Ac_i}-\eqref{eq:Maxwell_inc} together with the boundary conditions~\eqref{eq:bc} and the radiation condition~\eqref{eq:radiation} has a unique solution~\cite{KressRoach,KleinmanMartin}. The results in this text can be extended to the case of complex wavenumbers $k_i,i=1,2$, provided we assume uniqueness of the transmission problem and its adjoint.

\subsection{Layer integral potentials and operators \label{di_ind_cfie}}
A variety of integral equations for the transmission 
problem~\eqref{eq:Ac_i}-\eqref{eq:bc} 
exist~\cite{KressRoach,costabel-stephan,KleinmanMartin}. The starting point in 
the derivation of direct integral equations for transmission problems is the 
Green's identities. 
Hence, 
let 
\[
G_k(\mathbf{x}):=i/4 (k|\mathbf{x}|^{-1}/2\pi)^{(d-2)/2}H_{(d-2)/2}^{(1)}(k|\mathbf{x}|)
\]
the free space Green's functions corresponding to the Helmholtz equation with 
wavenumber
$k$. For the sake of a simpler exposition, from now on 
we will commit a slight abuse of notation and denote
\[
 G_j(\mathbf{x})=G_{k_j}(\mathbf{x}),\quad j=1,2.
\]
(The context will avoid any possible confusion). 

Next we define  the associated   single and double layer potential
\[
 [SL_k\varphi]({\bf z}):=\int_\Gamma G_k(\mathbf z - \mathbf y)\varphi(\mathbf 
y)d\sigma(\mathbf{y}),\quad 
[ DL_k\psi]({\bf z}):=\int_\Gamma  \frac{\partial G_k(\mathbf z - \mathbf 
y)}{\partial\mathbf{n}(\mathbf y)}\psi(\mathbf y)d\sigma(\mathbf{y})
\]
for $ {\bf z}\in \mathbb{R}^d\setminus \Gamma$. As before, 
$SL_j, DL_j$ denotes the layer potentials for the wavenumbers $k_j$, with 
$j=1,2$.

We have then the representation
formulas for the exterior and interior domain 
 \begin{equation}\label{green}
 u^1 =
 DL_1(\gamma_D^1 u^1) - SL_1(\gamma_N^1 u^1) ,\quad
u^2= 
- DL_2(\gamma_D^2 u^2) + SL_2(\gamma_N^2 u^2).
\end{equation}

%

In addition to Green's identities \eqref{green},
trace formulas of the single and double layer potential are needed in the 
derivation of integral equations for transmission problems. For a given 
wavenumber $k$, the traces on $\Gamma$ of the single and double layer 
potentials corresponding to the wavenumber $k$ and densities $\varphi$ and 
$\psi$ are given 
by
\begin{equation}
\label{traces}
\begin{array}{rclrcl} 
\gamma_D^1 SL_k(\varphi)&=&\gamma_D^2 SL_k(\varphi)=S_k\varphi  
&\gamma_N^j SL_k(\varphi)&=&\frac{(-1)^j}{2}\varphi+K_k^\top\varphi\quad j=1,2   
\\ 
\gamma_D^j DL_k(\psi)&=&\frac{(-1)^{j+1}}{2}\psi+K_k\psi\quad j=1,2 
&\gamma_N^1 DL_k(\psi)&=&\gamma_N^2 DL_k(\psi)=N_k\psi.
\end{array}
\end{equation}

In equations~\eqref{traces} the operators $K_k$ and $K^\top_k$  are the double 
layer and the adjoint of the double layer operator
defined for a given wavenumber $k$ and density $\varphi$ as
\begin{eqnarray*}
(K_k\varphi)(\mathbf x)&:=&\int_{\Gamma}\frac{\partial G_k(\mathbf x-\mathbf 
y)}{\partial\mathbf{n}(\mathbf y)}\varphi(\mathbf y)d\sigma(\mathbf y),\ \mathbf 
x\ {\rm on}\ \Gamma,\\
(K_k^\top\varphi)(\mathbf x)&:=&\int_{\Gamma}\frac{\partial G_k(\mathbf 
x-\mathbf y)}{\partial\mathbf{n}(\mathbf x)}\varphi(\mathbf y)d\sigma(\mathbf 
y),\ \mathbf x\ {\rm on}\ \Gamma.
\end{eqnarray*}
Furthermore,
\[
 (N_k \varphi)(\mathbf x)  :=  \text{FP} \int_\Gamma 
\frac{\partial^{2}G_k(\mathbf x -\mathbf y)}{\partial \mathbf{n}(\mathbf x) 
\partial \mathbf{n}(\mathbf y)} \varphi(\mathbf y)d\sigma(\mathbf y)
\]
is the so-called hypersingular operator (FP stands for 
the Hadamard Finite Part Integral). We point out 
\[
(N_k \varphi)(\mathbf x) 
=k^{2}\int_\Gamma G_k(\mathbf x -\mathbf y)
(\mathbf{n}(\mathbf x)\cdot\mathbf{n}(\mathbf y))\varphi(\mathbf 
y)d\sigma(\mathbf y)+ {\rm PV}
\int_\Gamma \partial_{s(\mathbf{x})} G_k(\mathbf x -\mathbf y)\partial_{s(\mathbf{y})} \varphi(\mathbf 
y)d\sigma(\mathbf y),\nonumber 
\]
when $d=2$, and 
\[
(N_k \varphi)(\mathbf x) 
=k^{2}\int_\Gamma G_k(\mathbf x -\mathbf y)
(\mathbf{n}(\mathbf x)\cdot\mathbf{n}(\mathbf y))\varphi(\mathbf 
y)d\sigma(\mathbf y)+ {\rm PV}
\int_\Gamma \overrightarrow{\rm curl}_\Gamma^{\mathbf x}\ G_k(\mathbf x -\mathbf 
y)\cdot\overrightarrow{\rm curl}_\Gamma^{\mathbf y}\ \varphi(\mathbf 
y)d\sigma(\mathbf y),\nonumber 
\]
when $d=3$, where PV denotes the Cauchy Principal value of the integral, 
$\partial_s$
the tangential derivative and
$\overrightarrow{\rm curl}_\Gamma\varphi=\nabla_\Gamma\varphi\times\mathbf{n}$. 
Finally, the single layer operator $S_k$ is defined as
\[
(S_k\varphi)(\mathbf x):=\int_\Gamma G_k(\mathbf x -\mathbf y)\varphi(\mathbf 
y)d\sigma(\mathbf y),\ \mathbf{x}\ {\rm on} \ \Gamma.
\]

Again, we will use $K_j,\ K_j^\top,\ N_j $ and $S_j$ for $j=1,2$ for 
denoting the layer operator associated to the wavenumbers $k_j$.

Having recalled the definition of the scattering boundary integral operators, 
we present next their mapping properties in appropriate Sobolev spaces of 
functions defined on the manifold $\Gamma$~\cite{mclean:2000,turc2}:

\begin{theorem}\label{regL} 
For a smooth curve/surface $\Gamma$ the mappings 
\begin{itemize}
\item $S_k:H^{s}(\Gamma)\to H^{s+1}(\Gamma)$
\item $N_k:H^{s+1}(\Gamma)\to H^{s}(\Gamma)$
\item $K_k^\top:H^{s}(\Gamma)\to H^{s+3}(\Gamma),\ d=2;\ 
K_k^\top:H^{s}(\Gamma)\to H^{s+1}(\Gamma),\ d=3$
\item $K_k:H^{s}(\Gamma)\to H^{s+3}(\Gamma),\ d=2;\ K_k:H^{s}(\Gamma)\to 
H^{s+1}(\Gamma),\ d=3$
\end{itemize}
are continuous
for all $s\in\mathbb{R}$. 
\end{theorem}

We will use throughout the text the following results about the smoothing 
properties of differences of boundary integral operators corresponding to 
different wavenumbers 

\begin{theorem}\label{thmrev2}
Let 
$\kappa_1$ and $\kappa_2$ be such that $\Re{\kappa_j}\geq 0$ and 
$\Im(\kappa_j)\geq 0$. Then, 
\begin{eqnarray}\label{eq:01:thmrev2}
S_{\kappa_1}-S_{\kappa_2}:H^{s}(\Gamma)\to H^{s+3}(\Gamma),\quad 
N_{\kappa_1}-N_{\kappa_2}:H^{s}(\Gamma)\to H^{s+1}(\Gamma)  
\end{eqnarray}
are continuous for any $s$. 

Moreover, for $d=3$,
\begin{equation}
\label{eq:03:thmrev2}
 K_{\kappa_1}-K_{\kappa_2}, \quad K_{\kappa_1}^\top -K^\top_{\kappa_2}:
 H^{s}(\Gamma)\to H^{s+2}(\Gamma)
\end{equation}
are also continuous for all $s$.
\end{theorem}
\begin{proof}
Let $B_R=\{x:\|x\|_2\leq R\}$ and set
 $D_R:=D_1\cap B_R$. We assume $R$ to be sufficiently large 
so that $D_2\subset \mathbb{R}^d\setminus D_R$. Fix $\varphi\in H^{s}(\Gamma)$ 
and denote
\[
v_{1}:={SL}_{\kappa_1}\varphi,\quad v_2:={SL}_{\kappa_2}\varphi.
\]
Observe that for $s>-1$ cf \cite[Cor 6.14]{mclean:2000}
\[
\|v_j\|_{H^{s+3/2}(D_R)}+
\|v_j\|_{H^{s+3/2}(D_2)}\le C \|\varphi\|_{H^{s}(\Gamma)},
\]
with $C>0$ independent of $\varphi$.
Then $\omega=v_1-v_2$ satisfies
\[
 \Delta \omega= {f},\quad\text{in $D_R\cup D_2$},\quad
 \gamma_N^1\omega-\gamma_N^2\omega=0,\quad \gamma_D^1\omega-\gamma_D^2\omega=0, 
\]
with $f=-\kappa_1^2 v_1-\kappa_2^2 v_2$. Observe that, due to \eqref{traces},
$f$ has a jump in the
normal derivative across $\Gamma$, no matter how smooth $\varphi$
is. Therefore, $\omega$ has a limited regularity in $\overline{D_R\cup D_2}$.  
However, we 
can apply \cite[Theorem 4.20]{mclean:2000} to prove that $\omega\in 
H^{s+7/2} (D_R\cup D_2)\cap H^2(\overline{D_R\cup D_2})$ and that there exists 
$C$,   again independent of $\varphi$, so that
\[
\|\omega\|_{H^{s+7/2} (D_R)}+
\|\omega\|_{H^{s+7/2} (D_2)}\le C\Big[\|f\|_{H^{s+3/2}(D_R)}+
\|f\|_{H^{s+3/2}(D_2)}\Big]\le   C' \|\varphi\|_{H^{s}(\Gamma)}.
\]
In other words, 
\[
 {SL}_{\kappa_1}-{SL}_{\kappa_2}:H^{s}(\Gamma)\to 
H^{s+7/2}(D_2),\quad {SL}_{\kappa_1}-{SL}_{\kappa_2}:H^{s}(\Gamma)\to 
H^{s+7/2}(D_R)
\]
are continuous for all $s>-1$ and any sufficiently
large $R$. 

Analogously, one can show  ${DL}_{\kappa_1}-{DL}_{\kappa_2}:H^{s+1}(\Gamma)\to 
H^{s+7/2}(D_2)$ and ${DL}_{\kappa_1}-{DL}_{\kappa_2}:H^{s+1}(\Gamma)\to 
H^{s+7/2}(D_R)$ are continuous for all $s>0$ and any sufficiently
large $R$. 

By \eqref{traces}, and the continuity of the trace and normal derivative in
the appropriate Sobolev spaces,
we can easily conclude  that 
\begin{align}
 S_{\kappa_1}-S_{\kappa_2}&:H^{s}(\Gamma)\to H^{s+3}(\Gamma),&
 N_{\kappa_1}-N_{\kappa_2}&:H^{s+1}(\Gamma)\to H^{s+2}(\Gamma)\nonumber\\
 K^\top_{\kappa_1}-K^\top_{\kappa_2}&:H^{s}(\Gamma)\to H^{s+2}(\Gamma),&
 K_{\kappa_1}-K_{\kappa_2}&:H^{s+1}(\Gamma)\to H^{s+3}(\Gamma),
 \label{eq:04:thmrev2}
\end{align}
are continuous for any $s>-1$.
A transposition argument in $L^2(\Gamma)$ proves that, again for $s>-1$, 
\begin{align}\label{eq:05:thmrev2}
 S_{\kappa_1}-S_{\kappa_2}&:H^{-s-3}(\Gamma)\to H^{-s}(\Gamma),&
 N_{\kappa_1}-N_{\kappa_2}&:H^{-s-2}(\Gamma)\to H^{-s-1}(\Gamma).\nonumber\\
K_{\kappa_1}-K_{\kappa_2}&:H^{-s-2}(\Gamma)\to H^{-s}(\Gamma),&
K^\top_{\kappa_1}-K^\top_{\kappa_2}&:H^{-s-3}(\Gamma)\to H^{-s-1}(\Gamma),
\end{align}
which with \eqref{eq:04:thmrev2}, proves \eqref{eq:01:thmrev2} for almost 
all  $s\in\mathbb{R}$. The remaining values are covered using the theory of 
interpolation in Sobolev spaces cf \cite[App. B]{mclean:2000}. Hence, for instance, 
 we  have already proved that
$ S_{\kappa_1}-S_{\kappa_2}:H^t(\Gamma)\to H^{t+3}(\Gamma)$ is continuous 
for $t\not\in[-2,-1]$.
For $s\in[-2,-1]$, 
we have, in the usual notation of 
interpolation spaces, $H^{s}(\Gamma)=[H^{s-2}(\Gamma),H^{s+2}(\Gamma)]_{1/2}$ and,
since  $s+2, s-2\not\in[-2,-1]$, we can make use of 
\eqref{eq:04:thmrev2}-\eqref{eq:05:thmrev2} and the theory of interpolation spaces to conclude
\[
  \|S_{\kappa_1}-S_{\kappa_2}\|_{H^{s}(\Gamma)\to H^{s+3}(\Gamma)}\le
  \|S_{\kappa_1}-S_{\kappa_2}\|^{1/2}_{H^{s+2}(\Gamma)\to H^{s+5}(\Gamma)}
  \|S_{\kappa_1}-S_{\kappa_2}\|^{1/2}_{H^{s-2}(\Gamma)\to H^{s+1}(\Gamma)}.
\]
The continuity for the excluded $s$ of
the remaining boundary layer operators are dealt in a similar way. 

\end{proof}

\section{Regularized combined source integral equations for transmission 
problems\label{cfier}}

We derive in this section regularized combined field integral equations for 
transmission problems that rely on the use of adequate approximations of the 
Dirichlet to Neumann (DtN) operators. We start by explaining the main idea of 
our strategy. To this end, we introduce several notations. For a field 
$\mathbf{u}=(u^1,u^2)^\top$ such that $u^1$ and $u^2$ are solutions of the 
equations~\eqref{eq:Ac_e} and~\eqref{eq:Ac_i} respectively, where in addition 
$u^1$ is radiative, we define by $\gamma_T(\mathbf{u})$ the operator that maps 
the field $\mathbf{u}$ to the boundary data of the transmission problems given 
in equations~\eqref{eq:bc}, that is 
\[
\gamma_T(\mathbf{u})=\left(\begin{array}{l}
\gamma_D^1u_1 -\gamma_D^2u_2\\\gamma_N^1u_1 -\nu \gamma_N^2u_2\end{array}\right).
\]
At the heart of our approach there are two operators $\mathcal{R}_1$ and 
$\mathcal{R}_2$ that for a given $\mathbf{u}$ defined as above map the {\it 
given} transmission boundary conditions (e.g. the operator 
$\gamma_T\mathbf{u}$) 
to the Cauchy data on $\Gamma$ given by the exterior and interior 
Dirichlet and Neumann traces of $\mathbf{u}$. More specifically, 
we  denote by $\mathcal{R}_1$ the matrix operator which maps the difference of 
exterior and interior Dirichlet and Neumann traces of the field $\mathbf{u}$ to  
the exterior Dirichlet and Neumann traces on $\Gamma$ of the component $u^1$ on 
$\Gamma$. We write this as 
\begin{equation}\label{eq:R1T}
\mathcal{R}_1\gamma_T(\mathbf{u})
=\gamma_C^1(\mathbf{u}),\quad\text{where}
\quad \gamma_C^1(\mathbf{u})=\left(\begin{array}{c}\gamma_D^1 
u^1\\ \gamma_N^1u^1\end{array}\right).
\end{equation}
Similarly, we denote by $\mathcal{R}_2$ 
the operator that maps the difference of exterior and interior Dirichlet and 
Neumann traces of the field $\mathbf{u}$ to the interior traces on $\Gamma$ of 
the component $u^2$, that is 
\begin{equation}\mathcal{R}_2\gamma_T(\mathbf{u}):=\gamma_C^2(\mathbf{u}) \quad 
\text{where}\quad \gamma_C^2(\mathbf{u})=\left(\begin{array}{c}\gamma_D^2u^2\\ 
\gamma_N^2u^2\end{array}\right).\end{equation}
On account of the boundary 
conditions in equations~\eqref{eq:bc}, it follows that 
$\mathcal{R}_2=\left(\begin{array}{cc}1&0\\0&\nu^{-1}\end{array}
\right)(\mathcal 
{R}_1-I)$, where $I$ denotes the identity matrix. The field $\mathbf{u}$ itself 
can be retrieved through the Green's formulas, i.e. equations~\eqref{green} 
from the Cauchy data 
$\gamma_C(\mathbf{u})=(\gamma_C^1(\mathbf{u})\ \gamma_C^2(\mathbf{u}))$ on 
$\Gamma$. We write this in operator form as
\[\mathbf{u}=\left(\begin{array}{cc}DL_1&-SL_1\\0&0\end{array}
\right)\left(\begin{array}{c}\gamma_D^1 
u^1\\\gamma_N^1u^1\end{array}\right) +\left(\begin{array}{cc}0 
&0\\-DL_2&SL_2\end{array}\right)\left(\begin{array}{c}\gamma_D^2 
u^2\\\gamma_N^2u^2\end{array}\right)
\]
or in short form as $\mathbf{u}=\mathcal{C}(\gamma_C\mathbf{u})$. Obviously, if 
we denote by $\mathcal{R}=(\mathcal{R}_1;\mathcal{R}_2)$, the following 
identity 
holds
\begin{equation}\label{eq:identity}
\gamma_T\mathcal{C}\mathcal{R}=I.
\end{equation}

The
operators $\mathcal{R}$ can be expressed in terms of Dirichlet to Neumann operators. However, their 
evaluation is numerically cumbersome, if at all possible. Our idea is to use 
instead certain suitable approximations $\widetilde{\mathcal{R}}_1$ of the 
operator $\mathcal{R}_1$ defined in equation~\eqref{matrixR}. Once operators 
$\widetilde{\mathcal{R}}_1$ are constructed, their counterparts can be taken to 
be 
$\mathcal{\widetilde{R}}_2=\left(\begin{array}{cc}1&0\\0&\nu^{-1}\end{array}
\right)(\mathcal{\widetilde{R}}_1-I)$. We look then for a field 
$\mathbf{u}=(u^1,u^2)^\top$ in the form 
$\mathbf{u}=\mathcal{C}\mathcal{\widetilde{R}}\mathbf{w}$, where 
$\mathbf{w}=(a,b)^\top$ is a vector density defined on $\Gamma$ and 
$\mathcal{\widetilde{R}}=(\mathcal{\widetilde{R}}_1;\mathcal{\widetilde{R}}
_2)$. More precisely, if we denote $\mathcal{\widetilde{R}}_1 = 
\left(\begin{array}{cc}\widetilde{R}_{11}&\widetilde{R}_{12}\\\widetilde{R}_{21}
&\widetilde{R}_{22}
\end{array}\right)$, we look for fields $u^1$ defined as
\begin{equation}
u^1(\mathbf{z})= 
DL_1(\widetilde{R}_{11}a+\widetilde{R}_{12}b)(\mathbf{z})-SL_1(\widetilde{R}_{21
}a+\widetilde{R}_{22} b)(\mathbf{z}),\quad 
\mathbf{z}\in\mathbb{R}^d\setminus\Gamma
\end{equation}
and $u^2$ defined as
\begin{equation}
u^2(\mathbf{z})= -DL_2(\widetilde{R}_{11} a+\widetilde{R}_{12} 
b-a)(\mathbf{z})+\nu^{-1}SL_2(\widetilde{R}_{21} a+\widetilde{R}_{22} 
b-b)(\mathbf{z}),\quad \mathbf{z}\in\mathbb{R}^d\setminus\Gamma.
\end{equation}
Using the jump conditions of the boundary layer potentials we are led to the 
integral equation 
\begin{equation}\label{eq:reg_diel}
\widetilde{\mathcal{D}}\left(\begin{array}{c} a\\ 
b\end{array}\right)=\gamma_T\mathcal{C}(\mathcal{\widetilde{R}}\mathbf{w}
)=-\left(\begin{array}{c}\gamma_D^1 u^{inc}\\\gamma_N^1 
u^{inc}\end{array}\right)
\end{equation}
which we refer to as Generalized Combined Source Integral Equation (GCSIE) and 
which takes on the following explicit form:
\begin{eqnarray}\label{Levadoux_simpl}
\left(\frac{1}{2}I-K_2+(K_1+K_2)\widetilde{R}_{11}-(S_1+\nu^{-1}
S_2)\widetilde { R } _ { 21 } \right)a\qquad&&\nonumber\\
+\left(\nu^{-1}S_2 + 
(K_1+K_2)\widetilde{R}_{12}-(S_1+\nu^{-1}S_2)\widetilde{R}_{22}
\right)b&=&-\gamma_D^1 u^{inc}\nonumber\\
\nonumber\\
\left(-\nu N_2 +(N_1+\nu 
N_2)\widetilde{R}_{11}-(K^\top_1+K^\top_2)\widetilde{R}_{21}
\right)a\qquad&&\nonumber\\
+\left(\frac{1}{2}I+K^\top_2+(N_1+\nu 
N_2)\widetilde{R}_{12}-(K^\top_1+K^\top_2)\widetilde{R}_{22}\right) 
b&=&-\gamma_N^1u^{inc}.
\end{eqnarray}
In what follows we compute the operator $\mathcal{R}$ in terms of Dirichlet to
Neumann operators and we establish in what sense should the operators 
$\mathcal{\widetilde{R}}$ approximate the operator $\mathcal{R}$ so that, in 
the light of formulas~\eqref{eq:identity} and~\eqref{eq:reg_diel}, the matrix 
operators $\mathcal{\widetilde{D}}$ in the left-hand side of GCSIE 
equations~\eqref{Levadoux_simpl} are close to the identity matrix. Notice that 
is precisely what we obtain when using ${R}_{ij}$ in \eqref{Levadoux_simpl} 
instead.

We present next a formal 
calculation of the operator $\mathcal{R}_1$ based on the use of 
Dirichlet-to-Neumann operators for the domains $D_j,j=1,2$. The latter 
operators 
are defined such that $Y^1$ maps the Dirichlet trace on the boundary $\Gamma$ 
of 
a radiative solution of the Helmholtz equation with wavenumber $k_1$ in the 
domain $D_1$ to its Neumann trace on the boundary $\Gamma$, i.e. 
$Y^1:\gamma_D^1\cdot\rightarrow \gamma_N^1\cdot$ and $Y^2$ maps the Dirichlet 
trace on the boundary $\Gamma$ of a solution of the Helmholtz equation with 
wavenumber $k_2$ in the domain $D_2$ to its Neumann trace on the boundary 
$\Gamma$, i.e. $Y^2:\gamma_D^2\cdot\rightarrow \gamma_N^2\cdot$. By \eqref{eq:R1T},
\begin{equation}\label{first_eq}
R_{11}(\gamma_D^1u^1-\gamma_D^2u^2)+R_{12}
(\gamma_N^1u^1-\nu\gamma_N^2u^2)=\gamma_D^1u^1.
\end{equation}
Equation~\eqref{first_eq} can be further expressed in term of the admittance 
operators $Y^j,j=1,2$ as the following (operator) linear system
\begin{eqnarray}\label{first_system}
R_{11}+R_{12}Y^1 &=& I\nonumber\\
R_{11}+\nu R_{12}Y^2 &=& 0.
\end{eqnarray}
We immediately obtain that the solution of the linear system in 
equation~(\ref{first_system}) is 
given by $R_{12}=(Y^1-\nu Y^2)^{-1}$ and $R_{11}=-\nu (Y^1-\nu Y^2)^{-1}Y^2$. 
Using the admittance operator $Y^1$ we see that the second row of the matrix 
operator $\mathcal{R}_1$ can be obtained by composing on the left the first row 
of $\mathcal{R}_1$ by $Y^1$. Hence we obtain
\begin{equation}\label{matrixR}
\mathcal{R}_1=\left(\begin{array}{cc}-\nu (Y^1-\nu Y^2)^{-1}Y^2 & (Y^1-\nu 
Y^2)^{-1}\\-\nu Y^1(Y^1-\nu Y^2)^{-1}Y^2 & Y^1(Y^1-\nu 
Y^2)^{-1}\end{array}\right).
\end{equation}
We note that the calculations that led to equation~\eqref{matrixR} are entirely 
formal, as the operators $(Y^1-\nu Y^2)^{-1}$ or $Y^2$ may not be 
well defined. Furthermore, the computation of the Dirichlet to Neumann 
operators $Y^j,j=1,2$ is expensive for general domains $D_j,j=1,2$. 
Nevertheless, operators $\widetilde{Y}^j,j=1,2$ can be constructed such that 
the 
difference operators $Y^j-\widetilde{Y}^j$ are smoother operators than $Y^j$ 
(these type of operators are referred to as parametrices). We discuss in what 
follows what degree of smoothing must the operators $Y^j-\widetilde{Y}^j,j=1,2$ 
have in order to lead to
GCSIE operators~\eqref{eq:reg_diel} that are Fredholm of the second kind in 
appropriate Sobolev spaces.


\section{Approximations of the admittance operators\label{app}}

Our goal is to 
produce appropriate approximations $\widetilde{\mathcal{R}}$ of the exact 
admittance operator $\mathcal{R}$ so that the matrix operators 
$\widetilde{\mathcal{D}}$ that enter GCSIE formulations~\eqref{eq:reg_diel} are 
(i) 
compact perturbations of the identity (matrix) operator in appropriate Sobolev 
spaces and (ii) invertible in the same spaces. 
We establish in this section sufficient conditions on the regularity properties 
of the difference matrix operators $\mathcal{\widetilde{R}}-\mathcal{R}$ that 
ensure the aforementioned property (i). We distinguish two cases with regards 
to Sobolev spaces the matrix operators $\widetilde{\mathcal{D}}$ act upon:

\noindent{\bf Case I} we consider $\gamma_D^1 u^{inc}\in H^s(\Gamma)$ and $\gamma_N^1 
u^{inc}\in 
H^s(\Gamma)$ which implies that the solution $(a,b)$ of the GCSIE 
formulations~\eqref{eq:reg_diel} has the same regularity, that is $(a,b)\in 
H^{s}(\Gamma)\times H^{s}(\Gamma)$;

\noindent{\bf Case II} we consider $\gamma_D^1 u^{inc}\in H^s(\Gamma)$ and $\gamma_N^1 
u^{inc}\in 
H^{s-1}(\Gamma)$ which implies $(a,b)\in 
H^{s}(\Gamma)\times H^{s-1}(\Gamma)$. 

Given that $\gamma_T\mathcal{C}\mathcal{R}=I$, we expect that once we construct 
operators $\widetilde{\mathcal{R}}$ with the desired properties (i) and (ii), 
the 
eigenvalues of the operators in the left-hand side of 
equation~(\ref{eq:reg_diel}) will accumulate at $(1,1)$. In addition, we strive 
to construct operators $\mathcal{\widetilde{R}}_1$ that are (iii) as simple as 
possible so that their evaluation is as numerically inexpensive as possible.

We present first a result that establishes in what sense should 
$\widetilde{\mathcal{R}}_1$ approximate $\mathcal{R}_1$ in order for the first 
property (i) to hold in {\bf Case I}. In order to make a more striking 
distinction between the two cases, 
we denote by $\mathcal{R}_1^{s,s}$ the approximating operators $\mathcal{R}_1$ 
of the operators $\mathcal{R}_1$ in the spaces $H^s(\Gamma)\times H^s(\Gamma)$. 
Given the mapping properties of the Dirichlet to Neumann 
operators $Y^j:H^{s}(\Gamma)\to H^{s-1}(\Gamma)$, and assuming that the 
operators $(Y^1-\nu Y^2)^{-1}$ are 
well defined, we have then the following mapping properties of the components 
$R_{ij},1\leq i,j\leq 2$ of the matrix operator $\mathcal{R}_1$ defined in 
equation~\eqref{matrixR}: $R_{11}:H^{s}(\Gamma)\to H^{s}(\Gamma)$, 
$R_{12}:H^{s}(\Gamma)\to H^{s+1}(\Gamma)$, $R_{21}:H^{s}(\Gamma)\to 
H^{s-1}(\Gamma)$, and $R_{22}:H^{s}(\Gamma)\to H^{s}(\Gamma)$. Our first 
important result is given in

\begin{theorem}\label{thm1}
Assume that the operators $(Y^1-\nu Y^2)^{-1}$ and $Y_2$ are well defined. 
Let $\widetilde{R}_{ij}^{s,s},i,j=1,2$ be operators such that for all 
$s\in\mathbb{R}$
we have
\begin{itemize}
\item $\widetilde{R}_{11}^{s,s}-R_{11}:H^{s}(\Gamma)\to H^{s+2}(\Gamma)$
\item $\widetilde{R}_{12}^{s,s}-R_{12}:H^{s}(\Gamma)\to H^{s+3}(\Gamma)$
\item $\widetilde{R}_{21}^{s,s}-R_{21}:H^{s}(\Gamma)\to H^{s+1}(\Gamma)$
\item $\widetilde{R}_{22}^{s,s}-R_{22}:H^{s}(\Gamma)\to H^{s+2}(\Gamma).$
\end{itemize}
Then the matrix operator $\widetilde{\mathcal{D}}^{s,s}$ 
defined in equation~\eqref{eq:reg_diel} and corresponding to the regularizing 
operator $\mathcal{R}_1^{s,s}$ has the following mapping property 
$\widetilde{\mathcal{D}}^{s,s}:H^{s}(\Gamma)\times H^{s}(\Gamma)\to 
H^{s}(\Gamma)\times 
H^{s}(\Gamma)$. 
Furthermore, the operator $\widetilde{\mathcal{D}}^{s,s}$ can be written in the 
form
$$\mathcal{\widetilde{D}}^{s,s}=I 
+\left(\begin{array}{cc}\widetilde{D}_{11}^r&\widetilde{D}_{12}^r\\ 
\widetilde{D}_{21}^r & \widetilde{D}_{22}^r\end{array}\right)$$
where $\widetilde{D}_{11}^r:H^{s}(\Gamma)\to 
H^{s+2}(\Gamma)$, $\widetilde{D}_{12}^r:H^{s}(\Gamma)\to H^{s+3}(\Gamma)$, 
$\widetilde{D}_{21}^r:H^{s}(\Gamma)\to H^{s+1}(\Gamma)$, and 
$\widetilde{D}_{22}^r:H^{s}(\Gamma)\to H^{s+2}(\Gamma)$. In particular, and 
given the compact 
embeddings of $H^{t}(\Gamma)$ into $H^{s}(\Gamma)$ for all $t>s$,  it follows 
that the matrix operator  $\mathcal{\widetilde{D}}^{s,s}$ is a compact 
perturbation of 
identity in the space  $H^{s}(\Gamma)\times H^{s}(\Gamma)\to 
H^{s}(\Gamma)\times H^{s}(\Gamma)$.
\end{theorem}

\begin{proof}
Given the assumptions about the operators $\widetilde{R}_{ij}^{s,s}$, it follows 
that 
they have the same mapping properties as the operators $R_{ij}$ for $i,j=1,2$, 
and thus the mapping property of the operators $\widetilde{\mathcal{D}}^{s,s}$ 
follow 
immediately. Let us consider first
the identity (see \eqref{eq:identity})
\[
  \mathcal{D}^{s,s}-I=\gamma_T{\cal C}\mathcal{R}^{s,s}-\gamma_T{\cal 
C}\mathcal{R}.
\]


The result of the 
Theorem follows once we show that $\widetilde{D}_{ij}^{s,s}-D_{ij}$ are 
regularizing 
operators of one order for all $i=1,2$ and $j=1,2$. A simple calculation gives
\[
\widetilde{D}_{11}^r=\widetilde{D}_{11}^{s,s}-I=\widetilde{D}_{11}^{s,s}-{D}_{
11}= (K_1+K_2)(\widetilde{R}^{s,s}_{11}-R_{11} )-(S_1+\nu^{-1 } 
S_2)(\widetilde{R}_{21}^{s,s}-R_{21}).
\]
We have that $\widetilde{R}_{11}^{s,s}-R_{11}:H^{s}(\Gamma)\to H^{s+2}(\Gamma)$ 
and 
$K_1+K_2: H^{s+2}(\Gamma)\to  H^{s+3}(\Gamma)$, together with 
$\widetilde{R}_{21}^{s,s}-R_{21}:H^{s}(\Gamma)\to H^{s+1}(\Gamma)$ and 
$S_1+\nu^{-1}S_2:H^{s+1}(\Gamma)\to H^{s+2}(\Gamma)$, and thus 
$\widetilde{D}_{11}^r:H^{s}(\Gamma)\to H^{s+2}(\Gamma)$. We also 
have
\[
 \widetilde{D}_{12}^r=\widetilde{D}_{12}^{s,s}=
\widetilde{D}_{12}^{s,s}-D_{12}=(K_1+K_2)(\widetilde{R}^{s,s}_{12}-R_{12}
)-(S_1+\nu^{-1}
S_2)(\widetilde {R}_{22}^{s,s}-R_{22})
\]
and hence $\widetilde{D}_{12}^r:H^{s}(\Gamma)\to H^{s+3}(\Gamma)$ since 
$\widetilde{R}_{12}^{s,s}-R_{12}:H^{s}(\Gamma)\to H^{s+3}(\Gamma)$, 
$K_1+K_2:H^{s+3}(\Gamma)\to H^{s+4}(\Gamma)$, 
$\widetilde{R}_{22}^{s,s}-R_{22}:H^{s}(\Gamma)\to H^{s+2}(\Gamma)$, and 
$S_1+\nu^{-1}S_2:H^{s+2}(\Gamma)\to H^{s+3}(\Gamma)$. Furthermore, given that
\[
 \widetilde{D}_{21}^r=\widetilde{D}_{21}^{s,s}=
\widetilde{D}_{21}^{s,s}-D_{21}=(N_1+\nu 
N_2)(\widetilde{R}_{11}^{s,s}-R_{11})-(K_1^\top+K_2^\top)(\widetilde{R}_{21}^{s,
s}-R_{21})
\]
we obtain that $\widetilde{D}_{21}^r:H^{s}(\Gamma)\to H^{s+1}(\Gamma)$ since 
$\widetilde{R}_{11}^{s,s}-R_{11}:H^{s}(\Gamma)\to H^{s+2}(\Gamma)$, $N_1+\nu 
N_2:H^{s+2}(\Gamma)\to H^{s+1}(\Gamma)$, 
$\widetilde{R}_{21}^{s,s}-R_{21}:H^{s}(\Gamma)\to 
H^{s+1}(\Gamma)$, and $K_1^\top+K^\top_2:H^{s+1}(\Gamma)\to H^{s+2}(\Gamma)$. 
Finally, 
we have that
\[ 
\widetilde{D}_{22}^r=\widetilde{D}_{22}^{s,s}-I=\widetilde{D}_{22}^{s,s}-D_{22}
=(N_1+\nu 
N_2)(\widetilde{R}_{12}^{s,s}-R_{12})-(K_1^\top+K_2^\top)(\widetilde{R}_{22}^{s,
s}-R_{22})
\]
from which we obtain that $\widetilde{D}_{22}^r:H^{s}(\Gamma)\to 
H^{s+2}(\Gamma)$ 
since $\widetilde{R}_{12}^{s,s}-R_{12}:H^{s}(\Gamma)\to H^{s+3}(\Gamma)$, 
$N_1+\nu 
N_2:H^{s+3}(\Gamma)\to H^{s+2}(\Gamma)$, 
$\widetilde{R}_{22}^{s,s}-R_{22}:H^{s}(\Gamma)\to 
H^{s+2}(\Gamma)$, and $K_1^\top+K^\top_2:H^{s+2}(\Gamma)\to H^{s+3}(\Gamma)$.
\end{proof}

\begin{remark}\label{remark1}
 We note that it follows immediately form the proof of Thorem~\ref{thm1} that the compactness result still holds when the requirements on the regularity of the difference operators $\widetilde{R}_{ij}^{s,s}-R_{ij},\ i,j=1,2$ are relaxed to $\widetilde{R}_{11}^{s,s}-R_{11}:H^{s}(\Gamma)\to H^{s+2}(\Gamma)$, $\widetilde{R}_{12}^{s,s}-R_{12}:H^{s}(\Gamma)\to H^{s+2}(\Gamma)$, $\widetilde{R}_{21}^{s,s}-R_{21}:H^{s}(\Gamma)\to H^{s}(\Gamma)$, and $\widetilde{R}_{22}^{s,s}-R_{22}:H^{s}(\Gamma)\to H^{s}(\Gamma).$ However, we will construct the operators $\widetilde{R}_{ij}^{s,s},i,j=1,2$ based on approximation $\widetilde{Y}^j$ of the Dirichlet-to-Neumann operators $Y^j$ for $j=1,2$, and thus we will start with constructing the simpler operator $\widetilde{R}_{12}^{s,s}$. Once $\widetilde{R}_{12}^{s,s}$ constructed, we will construct $\widetilde{R}_{11}^{s,s}=-\nu \widetilde{R}_{12}^{s,s}\widetilde{Y}^2$, and thus
\[
\widetilde{R}_{11}^{s,s}-R_{11}=-\nu(\widetilde{R}_{12}^{s,s}-R_{12})\widetilde{Y}^2-\nu R_{12}(\widetilde{Y}^2-Y^2).
\]
It follows then that in order for $\widetilde{R}_{11}^{s,s}-R_{11}:H^{s}(\Gamma)\to H^{s+2}(\Gamma)$--which is the optimal regularity property of the operator $\widetilde{R}_{11}^{s,s}-R_{11}$, we must have $\widetilde{Y}^2-Y^2:H^{s}(\Gamma)\to H^{s+1}(\Gamma)$ and $\widetilde{R}_{12}^{s,s}-R_{12}:H^{s}(\Gamma)\to H^{s+3}(\Gamma)$. We note that the latter regularity property of the operator $\widetilde{R}_{12}^{s,s}-R_{12}$ is what is required in Theorem~\ref{thm1}. Similar considerations motivate the other regularity requirements stated in Theorem~\ref{thm1}.
\end{remark}

We present next a result that establishes in what sense should 
$\widetilde{\mathcal{R}}_1$ approximate $\mathcal{R}_1$ in order for the first 
property (i) to hold in {\bf Case II}. In this case we denote by 
$\mathcal{R}_1^{s,s-1}$ the 
approximating operators $\mathcal{R}_1$ of the operators $\mathcal{R}_1$ in the 
spaces $H^s(\Gamma)\times H^{s-1}(\Gamma)$. Given the discussion in Remark~\ref{remark1}, we have

\begin{theorem}\label{thm10}
Assume that the operators $(Y^1-\nu Y^2)^{-1}$ and $Y_2$ are well defined. 
Let $\widetilde{R}_{ij}^{s,s-1},i,j=1,2$ be operators such that for all 
$s\in\mathbb{R}$
we have
\begin{itemize}
\item $\widetilde{R}_{11}^{s,s-1}-R_{11}:H^{s}(\Gamma)\to H^{s+1}(\Gamma)$
\item $\widetilde{R}_{12}^{s,s-1}-R_{12}:H^{s-1}(\Gamma)\to H^{s+1}(\Gamma)$
\item $\widetilde{R}_{21}^{s,s-1}-R_{21}:H^{s}(\Gamma)\to H^{s}(\Gamma)$
\item $\widetilde{R}_{22}^{s,s-1}-R_{22}:H^{s-1}(\Gamma)\to H^{s+1}(\Gamma).$
\end{itemize}
Then the matrix operator $\widetilde{\mathcal{D}}^{s,s-1}$ 
defined in equation~\eqref{eq:reg_diel} and corresponding to the regularizing 
operator $\mathcal{R}_1^{s,s-1}$
has the following mapping property 
$\widetilde{\mathcal{D}}^{s,s-1}:H^{s}(\Gamma)\times H^{s-1}(\Gamma)\to 
H^{s}(\Gamma)\times 
H^{s-1}(\Gamma)$. 
Furthermore, the operator $\widetilde{\mathcal{D}}^{s,s-1}$ can be written in 
the form
$$\mathcal{\widetilde{D}}^{s,s-1}=I 
+\left(\begin{array}{cc}\hat{D}_{11}^r&\hat{D}_{12}^r\\ 
\hat{D}_{21}^r & \hat{D}_{22}^r\end{array}\right)$$
where $\hat{D}_{11}^r:H^{s}(\Gamma)\to 
H^{s+1}(\Gamma)$, $\hat{D}_{12}^r:H^{s}(\Gamma)\to H^{s+2}(\Gamma)$, 
$\hat{D}_{21}^r:H^{s}(\Gamma)\to H^{s}(\Gamma)$, and 
$\hat{D}_{22}^r:H^{s}(\Gamma)\to H^{s+1}(\Gamma)$. In particular $\mathcal{\widetilde{D}}^{s,s-1}$ is a compact 
perturbation of 
identity in the space  $H^{s}(\Gamma)\times H^{s-1}(\Gamma)\to 
H^{s}(\Gamma)\times H^{s-1}(\Gamma)$.

\end{theorem}
\begin{proof}
The proof of this result follows the same lines as the proof in 
Theorem~\ref{thm1}, yet taking into account the different mapping properties of 
the operators $K_j$ and $K_j^\top$ in the case $d=3$.
\end{proof}

Having established sufficient conditions that the operators 
$\mathcal{\widetilde{R}}^{s,s}$ and $\mathcal{\widetilde{R}}^{s,s-1}$ should 
satisfy in order 
for the corresponding GCSIE defined in 
equations~\eqref{eq:reg_diel} be second kind Fredholm integral equations in the 
spaces
$H^s(\Gamma)\times H^s(\Gamma)$ and $H^s(\Gamma)\times H^{s-1}(\Gamma)$ 
respectively,
we  return to explicit constructions of the component operators 
$\widetilde{R}_{ij}^{s,s}$ and $\widetilde{R}_{ij}^{s,s-1}$. 
In the former case, the key ingredient in our construction of suitable operators 
 that satisfy the assumptions in 
Theorem~\ref{thm1} is the use of suitable approximations $\widetilde{Y}^j_{-1}$ 
of the 
Dirichlet to Neumann operators $Y^j$ such that 
$\widetilde{Y}^j_{-1}-Y^j:H^{s}(\Gamma)\to 
H^{s+1}(\Gamma)$ for $s\in\mathbb{R}$. Conversely, to construct 
 $\widetilde{R}_{ij}^{s,s-1}$ fulfilling  the hypothesis in 
Theorem~\ref{thm10} we make use instead of $\widetilde{Y}^j_{0}$ such that 
$\widetilde{Y}^j_{0}-Y^j:H^{s}(\Gamma)\to 
H^{s}(\Gamma)$ for $s\in\mathbb{R}$. First we present the following result

\begin{lemma}\label{lemma0}
Let $A, \widetilde{A}:H^s(\Gamma)\to H^{s+\alpha}(\Gamma)$ be invertible, with $
A-\widetilde{A}:H^s(\Gamma)\to H^{s+\beta}(\Gamma)$ for all $s$ with 
$\beta>\alpha$. Then $A^{-1}-\widetilde{A}^{-1} :H^s(\Gamma)\to 
H^{s-2\alpha+\beta}(\Gamma)$. 
In particular, $\widetilde{A}^{-1}$ is a compact 
perturbation of $A^{-1}$
\end{lemma}
\begin{proof}
The result follows readily from the identity
\[
 A^{-1}- \widetilde{A}^{-1}=A^{-1}( \widetilde{A}-A)\widetilde{A}^{-1}
\]
and the continuity properties of the operators involved. 
\end{proof}

\begin{lemma}\label{lemma1}
Let $\widetilde{Y}^j_{-1},j=1,2$ and $\widetilde{Y}^j_{0},j=1,2$ be operators 
such that 
$\widetilde{Y}^j_{-1}-Y^j:H^{s}(\Gamma)\to H^{s+1}(\Gamma)$ and respectively 
$\widetilde{Y}^j_{0}-Y^j:H^{s}(\Gamma)\to H^{s}(\Gamma)$
for $j=1,2$ and for all $s$. Assume 
that the operators $(Y^1-\nu Y^2)^{-1}$, $(\widetilde{Y}^1_{-1}-\nu 
\widetilde{Y}^2_{-1})^{-1}$, and  $(\widetilde{Y}^1_{0}-\nu 
\widetilde{Y}^2_{0})^{-1}$ are well defined. Then the following mapping 
properties 
$(\widetilde{Y}_{-1}^1-\nu \widetilde{Y}_{-1}^2)^{-1}-(Y^1-\nu 
Y^2)^{-1}:H^{s}(\Gamma)\to 
H^{s+3}(\Gamma)$ and $(\widetilde{Y}_{0}^1-\nu 
\widetilde{Y}_{0}^2)^{-1}-(Y^1-\nu Y^2)^{-1}:H^{s}(\Gamma)\to 
H^{s+2}(\Gamma)$ hold.
\end{lemma}
\begin{proof} 
The result follows from Lemma \ref{lemma0} by taking $A= Y^1-\nu Y^2$, 
$\widetilde{A}=\widetilde{Y}_{-1}^1-\nu \widetilde{Y}_{-1}^2$ with $\alpha=-1$ 
and $\beta=1$, and respectively by taking $A= Y^1-\nu Y^2$, 
$\widetilde{A}=\widetilde{Y}_{0}^1-\nu \widetilde{Y}_{0}^2$ with $\alpha=-1$ and 
$\beta=0$. 
\end{proof}


We turn next to the important issue of constructing operators $\widetilde{Y}^j$ 
such that $\widetilde{Y}^j-Y^j:H^{s}(\Gamma)\to H^{s+1}(\Gamma)$.\\

\section{Approximations of the exterior and interior Dirichet-to-Neumann
operators\label{appY}}

From
the Green's identities \eqref{green} and  the jump relations \eqref{traces} we
obtain
\begin{eqnarray*}
\gamma_D^1 u &=& \frac{1}{2}\gamma_D^1 u + K_1 \gamma_D^1 u -
S_1 \gamma_N^1 u\\
\gamma_N^1 u &=& N_1\gamma_D^1 u + \frac{1}{2}\gamma_N^1 u - 
K_1^\top \gamma_N^1 u.
\end{eqnarray*}
Thus 
\begin{eqnarray}
S_1 Y^1&=&-\frac{I}{2}+K_1 \label{eq:Y1:01}\\
\left(\frac{1}{2}I+K_1^\top\right)Y^1&=&N_1\label{eq:Y1:02}
\end{eqnarray}
Similarly, for the interior problem it holds
\begin{eqnarray}
S_2 Y^2&=&\frac{1}{2}I+K_2 \label{eq:Y2:01}\\
\left(-\frac{1}{2}I+K_2^\top\right)Y^2
&=&N_2 \label{eq:Y2:02}.
\end{eqnarray}

The last ingredient in our analysis are the following identities:
\begin{equation}
\label{eq:calderon}
 S_{\kappa} N_{\kappa}=-\frac{1}4 I+(K_{\kappa})^2,\quad
  N_{\kappa}S_{\kappa}=-\frac{1}4 I+(K_{\kappa}^\top)^2,
 \quad N_\kappa 
K_\kappa=K_\kappa^\top N_\kappa,\quad K_\kappa S_\kappa=S_\kappa K^\top_\kappa 
\end{equation}
which is a simple consequence of the Calder\'on identity and hold for any 
$\kappa$. The construction of operators $\widetilde{Y}^j_{-1},j=1,2$ and 
$\widetilde{Y}^j_{0},j=1,2$ 
is different in two-dimensions from three dimensions on account of the enhanced 
smoothing properties of the operators $K_\kappa$ and $K_\kappa^\top$ in two 
dimensions--see Theorem~\ref{regL}. Specifically, we have the following  results

\begin{lemma}\label{lemma3}
In the case $d=2$ the operator $\widetilde{Y}^1_{-1}=2N_\kappa$ has the desired 
property: 
$\widetilde{Y}^1_{-1}-Y^1:H^{s}(\Gamma)\to H^{s+1}(\Gamma)$ for all $\kappa$ 
such that $\Re\kappa\geq 0$ and $\Im\kappa\geq 0$.
\end{lemma}
\begin{proof}
Note that from \eqref{eq:Y1:02},
\[
 Y^1=2N_1-2K_1^\top Y^1
\]
and $K_1^\top:H^{s-1}(\Gamma)\to H^{s+2}(\Gamma)$ and $Y^1:H^s(\Gamma)\to 
H^{s-1}(\Gamma)$ are continuous. Thus $Y^1-2N_1:H^{s}(\Gamma)\to 
H^{s+2}(\Gamma)\to H^{s+1}(\Gamma)$ continuously. Given that 
$N_\kappa-N_1:H^s(\Gamma)\to H^{s+1}(\Gamma)$ (see Theorem~\ref{thmrev2}), the 
result now follows. 
\end{proof}

\begin{lemma}\label{lemma4}
In the case $d=2$ the operator $\widetilde{Y}^2_{-1}=-2N_\kappa$ has the desired 
property: 
$\widetilde{Y}^2_{-1}-Y^2:H^{s}(\Gamma)\to H^{s+1}(\Gamma)$ for all $\kappa$ 
such that $\Re\kappa\geq 0$ and $\Im\kappa\geq 0$.
\end{lemma}
\begin{proof}
The proof is similar to that of Lemma \ref{lemma3} by using 
\eqref{eq:Y2:02} instead. 
\end{proof}

\begin{lemma}\label{lemma33d}
In the case $d=3$ the operator $\widetilde{Y}^1_{0}=2N_\kappa$ has the desired 
property: 
$\widetilde{Y}^1_{0}-Y^1:H^{s}(\Gamma)\to H^{s}(\Gamma)$ for all $\kappa$ such 
that $\Re\kappa\geq 0$ and $\Im\kappa\geq 0$.
\end{lemma}
\begin{proof}
Note that from \eqref{eq:Y1:02},
\[
 Y^1=2N_1-2K_1^\top Y^1
\]
and $K_1^\top:H^{s-1}(\Gamma)\to H^{s}(\Gamma)$ and $Y^1:H^s(\Gamma)\to 
H^{s-1}(\Gamma)$ are continuous. 
Thus $Y^1-2N_1:H^{s}(\Gamma)\to H^{s}(\Gamma)$ continuously. Given that 
$N_\kappa-N_1:H^s(\Gamma)\to H^{s+1}(\Gamma)$ (see Theorem~\ref{thmrev2}), the 
result now follows. 
\end{proof}

\begin{lemma}\label{lemma43d}
In the case $d=2$ the operator $\widetilde{Y}^2_{0}=-2N_\kappa$ has the desired 
property: 
$\widetilde{Y}^2_{0}-Y^2:H^{s}(\Gamma)\to H^{s}(\Gamma)$ for all $\kappa$ such 
that $\Re\kappa\geq 0$ and $\Im\kappa\geq 0$.
\end{lemma}
\begin{proof}
The proof is similar to that of Lemma \ref{lemma33d} by using 
\eqref{eq:Y2:02} instead. 
\end{proof}

\begin{lemma}\label{lemma33dm1}
In the case $d=3$ the operator 
$\widetilde{Y}^1_{-1}=2N_{\kappa_1}-4N_{\kappa_2}K_{\kappa_2}$ has the desired 
property: 
$\widetilde{Y}^1_{-1}-Y^1:H^{s}(\Gamma)\to H^{s+1}(\Gamma)$ for all $\kappa_j$ 
such that $\Re\kappa_j\geq 0$ 
and $\Im\kappa_j\geq 0$ for $j=1,2$. Also, the operator 
$\widetilde{Y}^2_{-1}=-2N_{\kappa_1}-4N_{\kappa_2}K_{\kappa_2}$ has the desired 
property: 
$\widetilde{Y}^2_{-1}-Y^2:H^{s}(\Gamma)\to H^{s+1}(\Gamma)$ for all $\kappa_j$ 
such that $\Re\kappa_j\geq 0$ and $\Im\kappa_j\geq 0$ for $j=1,2$
\end{lemma}
\begin{proof}
Note that from \eqref{eq:Y1:02} and \eqref{eq:calderon}
\[
 Y^1=2N_1-2K_1^\top Y^1=2N_1-4K_1^\top N_1 +4 (K_1^\top)^2 Y^1 =
 2N_1-4 N_1K_1 +4 (K_1^\top)^2 Y^1  
\]
Since $K_1:H^{s}(\Gamma)\to H^{s+1}(\Gamma)$, $Y^1:H^{s+1}(\Gamma)\to 
H^{s}(\Gamma)$, and $K_1^\top:H^{s}(\Gamma)\to H^{s+1}(\Gamma)$ are continuous, 
we get that $Y^1-2N_1+4 N_1 K_1 :H^{s}(\Gamma)\to H^{s+1}(\Gamma)$ 
continuously. Using the fact that $N_1 K_1-N_{\kappa_2} 
K_{\kappa_2}=(N_1-N_{\kappa_2})K_1+N_{\kappa_2}(K_1-K_{\kappa_2})$, and given 
that $N_{\kappa_j}-N_1:H^s(\Gamma)\to H^{s+1}(\Gamma)$ and 
$K_1-K_{\kappa_2}:H^s(\Gamma)\to H^{s+2}(\Gamma)$  (see Theorem~\ref{thmrev2}), 
the first part of the theorem now follows. The second part of the theorem 
follows similarly.  
\end{proof}

Given the choices of operators $\widetilde{Y}^j_{-1}, j=1,2$ and 
$\widetilde{Y}^j_{0}, j=1,2$ presented in 
Lemmas~\ref{lemma3}--\ref{lemma33dm1}, we focus next on using Calder\'on's 
calculus to construct operators $\widetilde{R}_{ij}^{s,s}$ and $\widetilde{R}_{ij}^{s,s-1}$ which satisfy Theorem~\ref{thm1} and 
\ref{thm10} respectively.

\section{Construction of operators $\widetilde{R}_{ij}^{s,s}$ in the 
case $d=2$\label{case1}}

Note first
\begin{eqnarray}
\mathcal{R}_1&=&
\left(\begin{array}{cc}-\nu {R}_{12} {Y}^2 & {R}_{12}\\-\nu 
{Y}^1{R}_{12} {Y}^2 & 
{Y}^1{R}_{12}\end{array}\right),\quad {R}_{12}=(Y^1-\nu Y^2)^{-1}.
\end{eqnarray} 
We construct first the operator $\widetilde{R}_{12}^{s,s}$ which 
is the 
building block for the matrix operator $\widetilde{\mathcal{R}}_1^{s,s}$. We 
assume in all 
the results that follow that the operators $(Y^1-\nu Y^2)^{-1}$ are well 
defined. We begin with the following result:

\begin{lemma}\label{lemma5}
The operator 
\[
\widetilde {R}_{12}^{s,s}=-\frac{2}{1+\nu} S_{\kappa},\quad \text{with 
} \Re(\kappa), \Im(\kappa)\ge 0. 
\]
has the property 
$\widetilde{R}_{12}^{s,s}-R_{12}:H^{s}(\Gamma)\to H^{s+3}(\Gamma)$ for all $s$. 
\end{lemma}
\begin{proof}
Choosing $\Im\kappa>0$, we have $S_{\kappa}:H^s(\Gamma)\to H^{s+1}(\Gamma)$ is 
invertible. 
Then, from~\eqref{eq:Y1:01} we deduce
\begin{eqnarray*}
 Y^1&=& S_{\kappa}^{-1} S_{\kappa}Y^1=S^{-1}_{\kappa}S_1 Y^1 
+S_{\kappa}^{-1}( S_{\kappa}-S_1) Y^1\\
&=&
 -\frac{1}2 
S_{\kappa}^{-1}+ 
S_{\kappa}^{ -1 } K_{1}+S_{\kappa}^{-1}(S_{\kappa}-S_1)Y^1=-\frac{1}2 
S_{\kappa}^{-1} +L_1. 
\end{eqnarray*}
Applying Theorems \ref{regL}  and \ref{thmrev2} one easily checks that 
$L_1:H^s(\Gamma)\to H^{s+1}(\Gamma)$ is continuous. Similarly,
\[
 Y^2=\frac{1}2 
S_{\kappa}^{-1}+S_{\kappa}^{-1} 
K_{2}+S_{\kappa}^{-1}(S_{\kappa}-S_2)Y^2=\frac{1}2 
S_{\kappa}^{-1} +L_2, 
\]
with $L_2:H^s(\Gamma)\to H^{s+1}(\Gamma)$  being again continuous for all $s$. 
Then
\[
 Y^1-\nu Y^2= -\frac{  1+\nu}{2} S^{-1}_{\kappa}+ L_\nu
\]
with $L_\nu:H^s(\Gamma)\to H^{s+1}(\Gamma)$. The proof for 
this case is finished by applying Lemma \ref{lemma0} with $A= Y^1-\nu Y^2$, 
$\widetilde{A}= -\frac{  1+\nu}{2}  S^{-1}_{\kappa}$, $\alpha=-1$ and $\beta=1$ 
.

\end{proof}

\begin{lemma}\label{lemma7}
The operator 
\[
\widetilde{R}_{11}^{s,s}=\frac{\nu}{1+\nu}I
\] has the property 
$\widetilde{R}_{11}^{s,s}-R_{11}:H^{s}(\Gamma)\to H^{s+2}(\Gamma)$ for all 
$s$. 
\end{lemma}
\begin{proof} 
We note that
\begin{eqnarray*}
R_{11}&=&-\nu R_{12} Y^2=-\nu \widetilde{R}_{12}^{s,s}Y^2+\nu 
(\widetilde{R}_{12}^{s,s}-R_{12}) Y^2\\
&=&\frac{2\nu }{1+\nu} S_2 Y^2+\frac{2\nu}{1+\nu} (S_{\kappa}-S_2) Y^2 +\nu 
(\widetilde{R}_{12}^{s,s}-R_{12}) 
Y^2\\
&=& \frac{\nu}{1+\nu} I+\frac{2\nu}{1+\nu} K_2+\frac{2\nu}{1+\nu} 
(S_{\kappa}-S_2) Y^2 +\nu 
(\widetilde{R}_{12}^{s,s}-R_{12}) 
Y^2=:\frac{\nu}{1+\nu} I + L_{11}
\end{eqnarray*}
where we have applied Lemma \ref{lemma5} and \eqref{eq:Y2:01}. It is 
straightforward to check, using Theorems  \ref{regL} and  \ref{thmrev2} with  
Lemma \ref{lemma5}, that $L_{11}:H^s(\Gamma)\to H^{s+2}(\Gamma)$ is continuous. 
The result is now proven.
\end{proof}

\begin{lemma}\label{lemma8}
The operator $$\widetilde{R}_{21}^{s,s}=\frac{2\nu}{1+\nu}N_{\kappa}$$ has 
the property $\widetilde{R}_{21}^{s,s}-R_{21}:H^{s}(\Gamma)\to H^{s+1}(\Gamma)$ 
for all 
$s$. 
\end{lemma}
\begin{proof}
Observe that
\begin{eqnarray*}
R_{21}&=&Y^1R_{11}=-2N_1\widetilde{R}_{11}^{s,s}+(Y^1+2N_1)\widetilde{R}_{11}^{s
,s}+Y^1(R_{11}
-\widetilde{R}_{11}^{s,s})=-\frac{2\nu}{1+\nu} N_1 + L_{21}
\end{eqnarray*}
From Lemmas \ref{lemma3} and \ref{lemma7} we deduce that 
$L_{21}:H^{s}(\Gamma)\to H^{s+1}(\Gamma)$ is continuous. 

The proof is finished once we make use of Theorem \ref{thmrev2} which states 
that 
$N_1-N_{\kappa}:H^{s}(\Gamma)\to H^{s+1}(\Gamma)$ is also continuous.
\end{proof}

Finally, we establish the following result
\begin{lemma}\label{lemma9}
The operator 
\[
\widetilde{R}_{22}^{s,s}=\frac{1}{1+\nu}I 
\]
has the property 
$\widetilde{R}_{22}^{s,s}-R_{22}:H^{s}(\Gamma)\to H^{s+2}(\Gamma)$ for all 
$s$. 
\end{lemma}
\begin{proof} 
Note that
\begin{eqnarray*}
{R}_{22}&=&Y^1 R_{12}=2N_{\kappa} 
R_{12}+(Y^1-2N_{\kappa})R_{12}\\
&=&-\frac{4}{1+\nu} N_{\kappa}S_{\kappa}
+
2N_\kappa(R_{12}-\widetilde{R}_{12}^{s,s})+ 
(Y^1-2N_{1})R_{12} +2(N_{1}-N_{\kappa})R_{12}\\
&=&\frac{1}{1+\nu}I -\frac{4}{1+\nu}(K_{\kappa}^\top)^2+2N_\kappa(R_{12}-\widetilde{R}_{12}^{s,s})+ (Y^1-2N_{1})R_{12} 
+2(N_{1}-N_{\kappa})R_{12}\\
&=:&\frac{1}{1+\nu}I +L_{22}
\end{eqnarray*}
where we have used Lemma \ref{lemma5} and the second Calder\'on identity in 
\eqref{eq:calderon}. Since $L_{22}:H^{s}(\Gamma)\to 
H^{s+2}(\Gamma)$ is continuous, the proof is finished.
\end{proof}

In conclusion, the matrix operator $\mathcal{\widetilde{R}}_1^{s,s}$ takes on 
the form
\begin{equation}
\label{eq:tildR1}
\begin{array}{rclrcl} 
\widetilde{R}_{11}^{s,s}&=&\displaystyle\frac{\nu}{1+\nu}I,\quad&
\widetilde{R}_{12}^{s,s}&=&\displaystyle-\frac{2}{1+\nu}S_{\kappa}\\[1.25ex]
\widetilde{R}_{21}^{s,s}&=&\displaystyle\frac{2\nu}{1+\nu}N_{\kappa},\quad&
\widetilde{R}_{22}^{s,s}&=&\displaystyle\frac{1}{1+\nu}I,
\end{array}
\end{equation}
where $\kappa$ is a wavenumber such that $\Re(\kappa)\geq 0$ and $\Im(\kappa)> 
0$ satisfies the assumptions in Theorem~\ref{thm1}. 
Thus, the matrix operator that enters the GCSIE 
formulation~\eqref{Levadoux_simpl} that uses the operator 
$\mathcal{\widetilde{R}}_1^{s,s}$ defined in equation~\eqref{eq:tildR1} is a 
compact perturbation of the identity matrix in the space $H^s(\Gamma)\times 
H^s(\Gamma)$. We established the latter result on the assumption that the 
operators $(Y^1-\nu Y^2)^{-1}$ and $Y^2$ are well defined. This assumption was needed to 
provide a thorough justification of the steps that led to the construction of 
the operator $\mathcal{\widetilde{R}}_1^{s,s}$. It turns out that this 
assumption is not essential, as we prove

\begin{theorem}\label{thm2}
Let $d=2$ and let us denote by $\mathcal{\widetilde{D}}^{s,s}$ the operator in 
the left-hand-side of 
equation~\eqref{Levadoux_simpl} with $\mathcal{\widetilde{R}}_1$ being the 
operator 
$\mathcal{\widetilde{R}}_1^{s,s}$ defined in equation~\eqref{eq:tildR1}. Then 
the operator $\mathcal{\widetilde{D}}^{s,s}$ is a 
compact perturbation of the identity matrix in the space $H^{s}(\Gamma)\times 
H^{s}(\Gamma)$ for all $s$. 
\end{theorem}
\begin{proof} We make use of Calder\'on's identities~\eqref{eq:calderon} to 
express each of the components of the 
matrix operator $\mathcal{\widetilde{D}}^{s,s}$ in the form
\begin{eqnarray*}
\left(\begin{array}{cc}\widetilde{D}_{11}^{s,s}&\widetilde{D}_{12}^{s,s}
\\\widetilde{D}_{21}^{s,s}&\widetilde{D}_{22}^{s,s}\end{array}
\right)\left(\begin{array}{c}a\\b\end{array}\right)&=&-\left(\begin{array}{c}
\gamma_D^1 u^{inc}\\\gamma_N^1 u^{inc}\end{array}\right)
\end{eqnarray*}
where
\begin{eqnarray}
\widetilde{D}_{11}^{s,s}&=&I-\frac{1}{1+\nu}K_2+\frac{\nu}{1+\nu}K_1-\frac{2\nu}
{1+\nu}S_1(N_{\kappa_1}-N_1)-\frac{2\nu}{1+\nu}(K_1)^2\nonumber\\
&&-\frac{2}{1+\nu}S_2(N_{\kappa_1}-N_2)-\frac{2}{1+\nu}(K_2)^2\nonumber\\
\widetilde{D}_{12}^{s,s}&=&\frac{1}{1+\nu}(S_2-S_1)-\frac{2}{1+\nu}(K_1+K_2)S_{
\kappa_1}\nonumber\\
\widetilde{D}_{21}^{s,s}&=&\frac{\nu}{1+\nu}(N_1-N_2)-\frac{2\nu}{1+\nu}
(K_1^\top+K_2^\top)N_{\kappa_1}\nonumber\\
\widetilde{D}_{22}^{s,s}&=&I+\frac{\nu}{1+\nu}K_2^\top-\frac{1}{1+\nu}
K_1^\top-\frac{2}{1+\nu}(N_1-N_{\kappa})S_{\kappa}\nonumber\\
&&-\frac{2\nu}{1+\nu}(N_2-N_{\kappa})S_{\kappa}-2(K_{\kappa}^\top)^2.\label{eq:matrix_explicit}
\end{eqnarray}
The result now follows from the mapping properties recounted in 
Theorem~\ref{regL} and Theorem~\ref{thmrev2}.
\end{proof}

\section{Construction of operators $\widetilde{R}_{ij}^{s,s-1}$ in the 
case $d=3$.\label{case2}}  
We begin again with the construction of the operator 
$\widetilde{R}_{12}^{s,s-1}$ under the assumption that the operator
$(Y^1-\nu Y^2)^{-1}$ is well defined. We begin with the following result:

\begin{lemma}\label{lemma53dm1}
The operator 
\[
\widetilde {R}_{12}^{s,s-1}=-\frac{2}{1+\nu} S_{\kappa},\quad \text{with 
} \Re(\kappa), \Im(\kappa)> 0. 
\]
has the property 
$\widetilde{R}_{12}^{s,s-1}-R_{12}:H^{s}(\Gamma)\to H^{s+2}(\Gamma)$ for all 
$s$. 
\end{lemma}
\begin{proof}
Choosing $\Im\kappa>0$, we have $S_{\kappa}:H^s(\Gamma)\to H^{s+1}(\Gamma)$ is 
invertible. 
Then, just in the proof of Lemma~\ref{lemma5} we get $Y^1=-\frac{1}2 
S_{\kappa}^{-1} +L_1$, where $L_1$ has the same
definition as in Lemma~\ref{lemma5} with three dimensional boundary integral 
operators instead. Applying Theorems \ref{regL}  and \ref{thmrev2} one easily 
checks that 
$L_1:H^s(\Gamma)\to H^{s}(\Gamma)$ is continuous. Similarly $Y^2=\frac{1}2 
S_{\kappa}^{-1} +L_2$ with $L_2:H^s(\Gamma)\to H^{s}(\Gamma)$  being again 
continuous for all $s$. 
Then
\[
 Y^1-\nu Y^2= -\frac{  1+\nu}{2} S^{-1}_{\kappa}+ L_\nu
\]
with $L_\nu:H^s(\Gamma)\to H^{s}(\Gamma)$. The proof for 
this case is finished by applying Lemma \ref{lemma0} with $A= Y^1-\nu Y^2$, 
$\widetilde{A}= -\frac{  1+\nu}{2}  S^{-1}_{\kappa}$, $\alpha=-1$ and $\beta=0$.
\end{proof}

The next three results can be established along the same lines as the results in Lemma~\ref{lemma7}, Lemma~\ref{lemma8}, and Lemma~\ref{lemma9}:
\begin{lemma}\label{lemma73dm1}
The operator 
\[
\widetilde{R}_{11}^{s,s-1}=\frac{\nu}{1+\nu}I
\] has the property 
$\widetilde{R}_{11}^{s,s-1}-R_{11}:H^{s}(\Gamma)\to H^{s+1}(\Gamma)$ for all 
$s$. 
\end{lemma}
\begin{lemma}\label{lemma83dm1}
The operator $$\widetilde{R}_{21}^{s,s-1}=\frac{2\nu}{1+\nu}N_{\kappa}$$ has 
the property $\widetilde{R}_{21}^{s,s-1}-R_{21}:H^{s}(\Gamma)\to H^{s}(\Gamma)$ 
for all 
$s$. 
\end{lemma} 
\begin{lemma}\label{lemma93dm1}
The operator 
\[
\widetilde{R}_{22}^{s,s-1}=\frac{1}{1+\nu}I 
\]
has the property 
$\widetilde{R}_{22}^{s,s-1}-R_{22}:H^{s}(\Gamma)\to H^{s+1}(\Gamma)$ for all 
$s$. 
\end{lemma}

In conclusion, the matrix operator $\mathcal{\widetilde{R}}_1^{s,s-1}$ takes on 
the form
\begin{equation}\label{eq:tildR13dm1}
\begin{array}{rclrcl}
\widetilde{R}_{11}^{s,s-1}&=&\displaystyle\frac{\nu}{1+\nu}I\quad&
\widetilde{R}_{12}^{s,s-1}&=&\displaystyle -\frac{2}{1+\nu}S_{\kappa} \\[2ex]
\widetilde{R}_{21}^{s,s-1}&=&\displaystyle \frac{2\nu}{1+\nu}N_{\kappa}\quad&
\widetilde{R}_{22}^{s,s-1}&=&\displaystyle \frac{1}{1+\nu}I,
\end{array}
\end{equation}
where $\kappa$ is a wavenumber such that $\Re(\kappa)\geq 0$ and $\Im(\kappa)> 
0$ satisfies the assumptions in Theorem~\ref{thm10}. Thus, 
the matrix operator that enters the GCSIE formulation~\eqref{Levadoux_simpl} 
that uses the operator $\mathcal{\widetilde{R}}_1^{s,s-1}$ defined in 
equation~\eqref{eq:tildR13dm1} is a compact perturbation of the identity matrix 
in the space $H^{s}(\Gamma)\times H^{s-1}(\Gamma)$, provided that the operators 
$(Y^1-\nu Y^2)^{-1}$ are well defined. The latter assumption is not essential, 
as we establish next the analogue of the result in Theorem~\ref{thm2}: 
\begin{theorem}\label{thm23dm1}
Let $d=3$ and $\mathcal{\widetilde{D}}^{s,s-1}$ be the operator in the 
left-hand-side of 
equation~\eqref{Levadoux_simpl} with $\mathcal{\widetilde{R}}_1$ being the 
operator 
$\mathcal{\widetilde{R}}_1^{s,s-1}$ defined in equation~\eqref{eq:tildR13dm1}. 
Then the operator $\mathcal{\widetilde{D}}^{s,s-1}$ is a 
compact perturbation of the identity matrix in the space $H^{s}(\Gamma)\times 
H^{s-1}(\Gamma)$ for all $s$.
\end{theorem}
\begin{proof} We note that the operator $\mathcal{\widetilde{D}}^{s,s-1}$ is 
defined just as 
the operator $\mathcal{\widetilde{D}}^{s,s}$ in Theorem~\ref{thm2}, except that 
all the boundary integral operators that enter its definition are three 
dimensional analogues of the operators in the aforementioned theorem. The 
result now follows from equations~\eqref{eq:matrix_explicit} and from the 
mapping properties recounted in Theorem~\ref{regL} and Theorem~\ref{thmrev2}.
\end{proof}

Interestingly enough, if we view the operator $\mathcal{\widetilde{D}}^{s,s-1}$ as an operator in the space $H^s(\Gamma) \times H^s(\Gamma)$, we have the following result
\begin{theorem}\label{thm23dm1ss}
Let $d=3$ and $\mathcal{\widetilde{D}}^{s,s-1}$ be the operator in the 
left-hand-side of 
equation~\eqref{Levadoux_simpl} with $\mathcal{\widetilde{R}}_1$ being the 
operator 
$\mathcal{\widetilde{R}}_1^{s,s-1}$ defined in equation~\eqref{eq:tildR13dm1}. 
Then the operator $\mathcal{\widetilde{D}}^{s,s-1}$ is a 
compact perturbation of an operator that is invertible with a bounded inverse in the space $H^{s}(\Gamma)\times H^{s}(\Gamma)$ for all $s$. Thus, the operator $\mathcal{\widetilde{D}}^{s,s-1}$ is Fredholm in the space $H^{s}(\Gamma)\times H^{s}(\Gamma)$ for all $s$.
\end{theorem}
\begin{proof} It follows from equations~\eqref{eq:matrix_explicit} that
\[
\mathcal{\widetilde{D}}^{s,s-1}=\left(\begin{array}{cc}I & 0\\ -\frac{2\nu}{1+\nu}(K_1^\top+K_2^\top)N_{\kappa_1}& I \end{array}\right)+\mathcal{C}=\mathcal{P}+\mathcal{C}
\]
where the components of the matrix operator $\mathcal{C}=\left(\begin{array}{cc}\mathcal{C}_{11}&\mathcal{C}_{12}\\\mathcal{C}_{21}&\mathcal{C}_{22}\end{array}\right)$ are given by
\begin{eqnarray}
\mathcal{C}_{11}&=&-\frac{1}{1+\nu}K_2+\frac{\nu}{1+\nu}K_1-\frac{2\nu}
{1+\nu}S_1(N_{\kappa_1}-N_1)-\frac{2\nu}{1+\nu}(K_1)^2\nonumber\\
&&-\frac{2}{1+\nu}S_2(N_{\kappa_1}-N_2)-\frac{2}{1+\nu}(K_2)^2\nonumber\\
\mathcal{C}_{12}&=&\frac{1}{1+\nu}(S_2-S_1)-\frac{2}{1+\nu}(K_1+K_2)S_{
\kappa_1}\nonumber\\
\mathcal{C}_{21}&=&\frac{\nu}{1+\nu}(N_1-N_2)\nonumber\\
\mathcal{C}_{22}&=&\frac{\nu}{1+\nu}K_2^\top-\frac{1}{1+\nu}
K_1^\top-\frac{2}{1+\nu}(N_1-N_{\kappa})S_{\kappa}\nonumber\\
&&-\frac{2\nu}{1+\nu}(N_2-N_{\kappa})S_{\kappa}-2(K_{\kappa}^\top)^2.\nonumber
\end{eqnarray}
Taking into account the mapping properties recounted in Theorem~\ref{regL} and Theorem~\ref{thmrev2}, it follows that $\mathcal{C}:H^{s}(\Gamma)\times H^{s}(\Gamma)\to H^{s+1}(\Gamma)\times H^{s+1}(\Gamma)$, and thus the operator $\mathcal{C}$ is compact when viewed as an operator from the space $H^{s}(\Gamma)\times H^{s}(\Gamma)$ to itself. The matrix operator $\mathcal{P}$ is a continuous mapping in the space $H^{s}(\Gamma)\times H^{s}(\Gamma)$. Furthermore, we  have that
\[
\mathcal{P}^{-1}=\left(\begin{array}{cc}I & 0\\ \frac{2\nu}{1+\nu}(K_1^\top+K_2^\top)N_{\kappa_1}& I \end{array}\right),
\]   
which can be seen to be a continuous mapping in the same space $H^{s}(\Gamma)\times H^{s}(\Gamma)$.
\end{proof}

\section{Construction of operators $\widetilde{R}_{ij}^{s,s}$ in the 
case $d=3$\label{case3}}  
We begin again with the construction of the operator $\widetilde{R}_{12}^{s,s}$ 
under the assumption that the operator $(Y^1-\nu Y^2)^{-1}$ is well defined. We 
begin with the following result:

\begin{lemma}\label{lemma53d}
The operator 
\[
\widetilde {R}_{12}^{s,s}=-\frac{2}{1+\nu} 
S_{\kappa_1}-\frac{4(1-\nu)}{(1+\nu)^2}S_{\kappa_2}K_{\kappa_2}^\top,\quad 
\text{with 
} \Re(\kappa_j), \Im(\kappa_j)\ge 0,\ j=1,2,
\]
has the property 
$\widetilde{R}_{12}^{s,s}-R_{12}:H^{s}(\Gamma)\to H^{s+3}(\Gamma)$ for all $s$. 
\end{lemma}
\begin{proof}
We use the following results established in Lemma~\ref{lemma33dm1}, namely the 
operators $\widetilde{Y}^1_{-1}=2N_{\kappa_1}-4N_{\kappa_2}K_{\kappa_2}$ and 
$\widetilde{Y}^2_{-1}=-2N_{\kappa_1}-4N_{\kappa_2}K_{\kappa_2}$ are such that 
$\widetilde{Y}^j_{-1}-Y^j:H^s(\Gamma)\to H^{s+1}(\Gamma)$. 

Let $A=R_{12}= Y^1-\nu Y^2$, 
$\widetilde{A}= 2(1+\nu)N_{\kappa_1}-4(1-\nu)N_{\kappa_2}K_{\kappa_2}$ and
assume first that $\widetilde{A}:H^s(\Gamma)\to H^{s-1}(\Gamma)$ is invertible. 
Applying Lemma \ref{lemma0} we conclude that 
$R_{12}-\widetilde{A}^{-1}:H^s(\Gamma)\to H^{s+3}(\Gamma)$. We also have
\begin{eqnarray}
\widetilde{R}_{12}^{s,s}\widetilde{A}&=&-4S_{\kappa_1}N_{\kappa_1}+\frac{
8(1-\nu)}{1+\nu}S_{\kappa_1}N_{\kappa_2}K_{\kappa_2}-\frac{8(1-\nu)}{1+\nu}S_{
\kappa_2}K_{\kappa_2}^\top N_{\kappa_1}\nonumber\\
&&+\frac{16(1-\nu)^2}{(1+\nu)^2}S_{\kappa_2}K_{\kappa_2}^\top 
N_{\kappa_2}K_{\kappa_2}\nonumber\\
&=&I-(K_{\kappa_1})^2+\frac{8(1-\nu)}{1+\nu}\left((S_{\kappa_1}-S_{\kappa_2})K_{
\kappa_2}^\top N_{\kappa_2}+S_{\kappa_2}K_{\kappa_2}^\top 
(N_{\kappa_2}-N_{\kappa_1})\right)\nonumber\\
&&+\frac{16(1-\nu)^2}{(1+\nu)^2}S_{\kappa_2}K_{\kappa_2}^\top 
N_{\kappa_2}K_{\kappa_2}:=I+M_\nu.
\end{eqnarray}
Making use of Theorem~\ref{regL} and Theorem~\ref{thmrev2} we obtain that 
$M_\nu:H^{s+1}(\Gamma)\to H^{s+3}(\Gamma)$ 
continuously. Thus, 
$\widetilde{R}_{12}^{s,s}\widetilde{A}-I=(\widetilde{R}_{12}^{s,s}-\widetilde{A}
^{-1})\widetilde{A}:H^{s+1}(\Gamma)\to H^{s+3}(\Gamma)$ and hence 
$\widetilde{R}_{12}^{s,s}-\widetilde{A}^{-1}:H^s(\Gamma)\to H^{s+3}(\Gamma)$. 

Using $\widetilde{R}_{12}^{s,s}-{R}_{12}^{s,s}=(\widetilde{R}_{12}^{s,s}-\widetilde{A}^{-1})
+(\widetilde{A}^{-1}-R_{12}^{s,s})$ the result of the lemma now follows .

If $\widetilde{A}$ fails to be  invertible, we can apply the same argument with different
wavenumbers ${\widetilde{k}_1,\widetilde{k_2}}$ for which the corresponding operator
is invertible (it is easy to show, using Lemma \ref{lemma:A01}, that if suffices to take 
$\widetilde{k}_1,\widetilde{k_2}$ with zero real part) and next apply Theorem \ref{thmrev2} to replace the layer operators with 
the original ones $k_1$, $k_2$ having in mind that the difference is a smoothing operator
of enough order.
\end{proof}

\begin{lemma}\label{lemma63d}
The operator 
\[
\widetilde 
{R}_{11}^{s,s}=\frac{\nu}{1+\nu}\left(I+\frac{4}{1+\nu}K_{\kappa_2}\right) 
,\quad \text{with 
} \Re(\kappa_2), \Im(\kappa_2)\ge 0,
\]
has the property 
$\widetilde{R}_{11}^{s,s}-R_{11}:H^{s}(\Gamma)\to H^{s+2}(\Gamma)$ for all $s$. 
\end{lemma}
\begin{proof} We have that
\begin{eqnarray*}
R_{11}&=&\nu(\widetilde{R}_{12}^{s,s}-R_{12})Y^2-\nu \widetilde{R}_{12}^{s,s} 
(Y^2-\widetilde{Y}^2_{-1})\nonumber\\
&&-4\nu\left(\frac{1}{1+\nu}S_{\kappa_1}+\frac{2(1-\nu)}{(1+\nu)^2}S_{\kappa_2}
K_{\kappa_2}^\top\right)(N_{\kappa_1}+2N_{\kappa_2}K_{\kappa_2})\nonumber\\
&=&\widetilde {R}_{11}^{s,s}+M_{11}
\end{eqnarray*}
with
\begin{eqnarray*}
M_{11}&:=&\nu(\widetilde{R}_{12}^{s,s}-R_{12})Y^2-\nu \widetilde{R}_{12}^{s,s} 
(Y^2-\widetilde{Y}^2_{-1})\nonumber\\
&&-\frac{8\nu(1-\nu)}{(1+\nu)^2}S_{\kappa_2} K_{\kappa_2}^\top (N_{\kappa_1}-N_{\kappa_2})
-\frac{16\nu}{(1+\nu)^2}(K_{
\kappa_2})^3
-\frac{4\nu}{1+\nu}(K_{\kappa_1})^2\nonumber\\
&&-\frac{16\nu(1-\nu)}{(1+\nu)^2}S_{\kappa_2} 
N_{\kappa_2}K_{\kappa_2}^2.
\end{eqnarray*}
Making use of Theorem~\ref{regL} and Theorem~\ref{thmrev2} we obtain that 
$M_{11}:H^{s}(\Gamma)\to H^{s+2}(\Gamma)$ continuously, 
and the result of the lemma follows.
\end{proof}

\begin{lemma}\label{lemma83d}
The operator 
$$\widetilde{R}_{21}^{s,s}=\nu\left(\frac{2}{1+\nu}N_{\kappa_1}+\frac{4(1-\nu)}{
(1+\nu)^2}N_{\kappa_2}K_{\kappa_2}\right)$$ has 
the property $\widetilde{R}_{21}^{s,s}-R_{21}:H^{s}(\Gamma)\to H^{s+1}(\Gamma)$ 
for all 
$s$. 
\end{lemma}
\begin{proof}
Observe that
\begin{eqnarray*}
R_{21}&=&Y^1R_{11}=Y^1(R_{11}-\widetilde{R}_{11}^{s,s})+(Y^1-\widetilde{Y}^1_{-1})\widetilde
{R}_{11}^{s,s}+\widetilde{Y}^1_{-1}\widetilde{R}_{11}^{s,s}\nonumber\\
&=&\widetilde{R}_{21}^{s,s}+M_{21}\nonumber\\
M_{21}&:=&Y^1(R_{11}-\widetilde{R}_{11}^{s,s})+(Y^1-\widetilde{Y}^1_{-1})\widetilde{R}_{11}^
{s,s}+\frac{8\nu}{(1+\nu)^2}(N_{\kappa_1}-N_{\kappa_2})K_{\kappa_2}
-\frac{16\nu}{(1+\nu)^2}N_{\kappa_2}(K_{\kappa_2})^2.
\end{eqnarray*}
Since
$M_{21}:H^{s}(\Gamma)\to H^{s+1}(\Gamma)$ 
continuously, see Theorem~\ref{regL} and Theorem~\ref{thmrev2}, the Lemma is proven.
\end{proof}

Finally, we establish the following result
\begin{lemma}\label{lemma93d}
The operator 
\[
\widetilde{R}_{22}^{s,s}=\frac{1}{1+\nu}I 
-\frac{4\nu}{(1+\nu)^2}K_{\kappa_2}^\top
\]
has the property 
$\widetilde{R}_{22}^{s,s}-R_{22}:H^{s}(\Gamma)\to H^{s+2}(\Gamma)$ for all 
$s$. 
\end{lemma}
\begin{proof} 
Observe that
\begin{eqnarray*}
R_{22}&=&Y^1R_{12}=Y^1(R_{12}-\widetilde{R}_{12}^{s,s})+(Y^1-\widetilde{Y}^1_{-1})\widetilde
{R}_{12}^{s,s}+\widetilde{Y}^1_{-1}\widetilde{R}_{12}^{s,s}\nonumber\\
&=&\widetilde{R}_{22}^{s,s}+M_{22}\nonumber\\
M_{22}&:=&Y^1(R_{12}-\widetilde{R}_{12}^{s,s})+(Y^1-\widetilde{Y}^1_{-1})\widetilde{R}_{12}^
{s,s}-\frac{4}{1+\nu}(K_{\kappa_2}^\top)^2\nonumber\\
&&-\frac{8(1-\nu)}{(1+\nu)^2}(N_{\kappa_1}-N_{\kappa_2})S_{\kappa_2}K_{\kappa_2}
^\top+\frac{16\nu}{(1+\nu)^2}(K_{\kappa_2}^\top)^3+\frac{16(1-\nu)}{(1+\nu)^2}N_{\kappa_2}K_{\kappa_2}S_{\kappa_2}K_{\kappa_2}
^\top.
\end{eqnarray*}
Making use of Theorem~\ref{regL} and Theorem~\ref{thmrev2} we obtain that 
$M_{22}
:H^{s}(\Gamma)\to H^{s+2}(\Gamma)$ continuously, and the result of the 
lemma follows.
\end{proof}

In conclusion, the matrix operator $\mathcal{\widetilde{R}}_1^{s,s}$ takes on 
the form
\begin{equation}\label{eq:tildR13d}
\begin{array}{rclrcl}
\widetilde{R}_{11}^{s,s}&=&\displaystyle\frac{\nu}{1+\nu}I+\frac{4\nu}{(1+\nu)^2}K_{\kappa_2}\quad&
\widetilde{R}_{12}^{s,s}&=&\displaystyle-\frac{2}{1+\nu} 
S_{\kappa_1}-\frac{4(1-\nu)}{(1+\nu)^2}S_{\kappa_2}K_{\kappa_2}^\top \\[2ex]
\widetilde{R}_{21}^{s,s}&=&\displaystyle\frac{2\nu}{1+\nu}N_{\kappa_1}+\frac{4\nu(1-\nu)}{
(1+\nu)^2}N_{\kappa_2}K_{\kappa_2}\quad&
\widetilde{R}_{22}^{s,s}&=&\displaystyle\frac{1}{1+\nu}I 
-\frac{4\nu}{(1+\nu)^2}K_{\kappa_2}^\top,
 \end{array}
\end{equation}
where $\kappa_j$ are wavenumbers such that $\Re(\kappa_j)\geq 0$ and 
$\Im({\kappa_j})> 0$ for $j=1,2$ 
satisfies the assumptions in Theorem~\ref{thm1}. Thus, the matrix operator that 
enters the GCSIE formulation~\eqref{Levadoux_simpl} that uses the operator 
$\mathcal{\widetilde{R}}_1^{s,s}$ defined in equation~\eqref{eq:tildR13d} is a 
compact perturbation of the identity matrix in the space $H^{s}(\Gamma)\times 
H^{s}(\Gamma)$, provided that the operators $(Y^1-\nu Y^2)^{-1}$ and $Y_2$ are well 
defined. The latter assumption is not essential, as we establish next the 
analogue of the result in Theorem~\ref{thm2}: 
\begin{theorem}\label{thm23d}
Let $d=3$ and $\mathcal{\widetilde{D}}^{s,s}$ be the operator in the 
left-hand-side of 
equation~\eqref{Levadoux_simpl} with $\mathcal{\widetilde{R}}_1$ being the 
operator 
$\mathcal{\widetilde{R}}_1^{s,s}$ defined in equation~\eqref{eq:tildR13d}. Then 
the operator $\mathcal{\widetilde{D}}^{s,s}$ is a 
compact perturbation of the identity matrix in the space $H^{s}(\Gamma)\times 
H^{s}(\Gamma)$ for all $s$.
\end{theorem}
\begin{proof} We make use of Calder\'on's identities~\eqref{eq:calderon} to 
express each of the components of the matrix operator 
$\mathcal{\widetilde{D}}^{s,s}$ in the case $d=3$ in the form
\begin{eqnarray}\label{eq:matrix_explicit3d}
\widetilde{D}_{11}^{s,s}&=&I-\frac{1}{1+\nu}K_2+\frac{\nu}{1+\nu}K_1+\frac{4\nu}
{(1+\nu)^2}(K_1+K_2)K_{\kappa_2}-2(K_{\kappa_1})^2\nonumber\\
&&-\frac{2\nu}{1+\nu}\left((S_1-S_{\kappa_1})+\nu^{-1}(S_2-S_{\kappa_1}
)\right)N_{ 
\kappa_1}-\frac{4\nu(1-\nu)}{(1+\nu)^2}(S_1+\nu^{-1}S_2)N_{\kappa_2}K_{\kappa_2}
 \nonumber\\
\widetilde{D}_{12}^{s,s}&=&\frac{1}{1+\nu}(S_2-S_1)-\frac{2}{1+\nu}
(K_1+K_2)\left(S_{\kappa_1}+\frac{2(1-\nu)}{1+\nu}S_{\kappa_2}K_{\kappa_2}
^\top\right)\nonumber\\
&&+\frac{4\nu}{(1+\nu)^2}(S_1+\nu^{-1}S_2)K_{\kappa_2}^\top\nonumber\\
\widetilde{D}_{21}^{s,s}&=&\frac{\nu}{1+\nu}(N_1-N_2)+\frac{4\nu}{(1+\nu)^2}
\left((N_1-N_{\kappa_2})+\nu(N_2-N_{\kappa_2})\right)K_{\kappa_2}\nonumber\\
&&-\frac{2\nu}{1+\nu}(K_1^\top+K_2^\top)(N_{\kappa_1}-N_{\kappa_2})-\frac{2\nu}{
1+\nu}(K_1^\top-K_{\kappa_2}^\top+K_2^\top-K_{\kappa_2}^\top)N_{\kappa_2}
\nonumber\\
&&-\frac{4\nu(1-\nu)}{(1+\nu)^2}(K_1^\top+K_2^\top)N_{\kappa_2}K_{\kappa_2}
\nonumber\\
\widetilde{D}_{22}^{s,s}&=&I+\frac{\nu}{1+\nu}K_2^\top-\frac{1}{1+\nu}
K_1^\top-\frac{2}{1+\nu}\left(N_1-N_{\kappa_1}+\nu(N_2-N_{\kappa_1})\right)S_{
\kappa_1}\nonumber\\
&&-\frac{4(1-\nu)}{(1+\nu)^2}(N_1+\nu 
N_2)S_{\kappa_2}K_{\kappa_2}^\top+\frac{4\nu}{(1+\nu)^2}(K_1^\top+K_2^\top)K_{
\kappa_2}^\top-2(K_{\kappa_1}^\top)^2.
\end{eqnarray}
The result now follows from the mapping properties recounted in 
Theorem~\ref{regL} and Theorem~\ref{thmrev2}.
\end{proof}

Interestingly, if we consider the following modified version of the operator $\widetilde{\mathcal{R}}_1^{s,s}$ given by a new simplified operator $\widetilde{\mathcal{R}}_1^{1,s,s}$ whose entries are defined by
\begin{equation}\label{eq:tildR13dN}
\begin{array}{rclrcl}
\widetilde{R}_{11}^{1,s,s}&=&\displaystyle\frac{\nu}{1+\nu}I+\frac{4\nu}{(1+\nu)^2}K_{\kappa_2}\quad&
\widetilde{R}_{12}^{1,s,s}&=&\displaystyle-\frac{2}{1+\nu} 
S_{\kappa_1} \\[2ex]
\widetilde{R}_{21}^{1,s,s}&=&\displaystyle\frac{2\nu}{1+\nu}N_{\kappa_1}\quad&
\widetilde{R}_{22}^{1,s,s}&=&\displaystyle\frac{1}{1+\nu}I 
-\frac{4\nu}{(1+\nu)^2}K_{\kappa_2}^\top,
 \end{array}
\end{equation}
where $\kappa_j$ are wavenumbers such that $\Re(\kappa_j)\geq 0$ and 
$\Im({\kappa_j})> 0$ for $j=1,2$, then it can be immediately seen that the operator $\widetilde{\mathcal{R}}_1^{1,s,s}$ does satisfy the relaxed smoothing requirements in Remark~\ref{remark1}. We establish 
\begin{theorem}\label{thm23dN}
Let $d=3$ and $\mathcal{\widetilde{D}}^{1,s,s}$ be the operator in the 
left-hand-side of 
equation~\eqref{Levadoux_simpl} with $\mathcal{\widetilde{R}}_1$ being the 
operator 
$\mathcal{\widetilde{R}}_1^{1,s,s}$ defined in equation~\eqref{eq:tildR13dN}. Then 
the operator $\mathcal{\widetilde{D}}^{1,s,s}$ is a 
compact perturbation of the identity matrix in the space $H^{s}(\Gamma)\times 
H^{s}(\Gamma)$ for all $s$.
\end{theorem}
\begin{proof} We make use of Calder\'on's identities~\eqref{eq:calderon} to 
express each of the components of the matrix operator 
$\mathcal{\widetilde{D}}^{1,s,s}$ in the case $d=3$ in the form
\begin{eqnarray*}
\widetilde{D}_{11}^{1,s,s}&=&I-\frac{1}{1+\nu}K_2+\frac{\nu}{1+\nu}K_1+\frac{4\nu}
{(1+\nu)^2}(K_1+K_2)K_{\kappa_2}-2(K_{\kappa_1})^2\nonumber\\
&&-\frac{2\nu}{1+\nu}\left((S_1-S_{\kappa_1})+\nu^{-1}(S_2-S_{\kappa_1}
)\right)N_{ 
\kappa_1}
 \nonumber\\
\widetilde{D}_{12}^{1,s,s}&=&\frac{1}{1+\nu}(S_2-S_1)-\frac{2}{1+\nu}
(K_1+K_2) S_{\kappa_1}+\frac{4\nu}{(1+\nu)^2}(S_1+\nu^{-1}S_2)K_{\kappa_2}^\top\nonumber\\
\widetilde{D}_{21}^{1,s,s}&=&\frac{\nu}{1+\nu}(N_1-N_2)+\frac{4\nu}{(1+\nu)^2}
\left((N_1-N_{\kappa_2})+\nu(N_2-N_{\kappa_2})\right)K_{\kappa_2}\nonumber\\
&&-\frac{2\nu}{1+\nu}(K_1^\top+K_2^\top)(N_{\kappa_1}-N_{\kappa_2})-\frac{2\nu}{
1+\nu}(K_1^\top-K_{\kappa_2}^\top+K_2^\top-K_{\kappa_2}^\top)N_{\kappa_2}
\nonumber\\
\widetilde{D}_{22}^{1,s,s}&=&I+\frac{\nu}{1+\nu}K_2^\top-\frac{1}{1+\nu}
K_1^\top-\frac{2}{1+\nu}\left(N_1-N_{\kappa_1}+\nu(N_2-N_{\kappa_1})\right)S_{
\kappa_1}-2(K_{\kappa_1}^\top)^2\nonumber\\
&+&\frac{4\nu}{(1+\nu)^2}(K_1^\top+K_2^\top)K_{\kappa_2}^\top.
\end{eqnarray*}
We easily see that $\mathcal{\widetilde{D}}^{1,s,s}-I:H^s(\Gamma)\times H^s(\Gamma)\to
H^{s+1}(\Gamma)\times H^{s+1}(\Gamma)$, from which the result of the Theorem now follows.
\end{proof}

\section{Well-posedness of the Generalized Combined Source Integral 
Equations\label{uniq}}

Having proved in Theorem~\ref{thm2}, Theorem~\ref{thm23dm1}, and 
Theorem~\ref{thm23d} that the various operators $\widetilde{\mathcal{D}}$ are 
Fredholm of the second kind in appropriate Sobolev spaces, we establish in this 
section the invertibility of those operators under certain conditions on the 
wavenumbers $\kappa_1$ and $\kappa_2$ that enter the definitions of the various 
regularizing operators $\widetilde{\mathcal{R}}_1$. We note that the 
regularizing operators $\widetilde{\mathcal{R}}_1$ defined in 
equations~\eqref{eq:tildR1},~\eqref{eq:tildR13dm1},~\eqref{eq:tildR13d}, and~\eqref{eq:tildR13dN} can 
be all defined componentwise by 
\begin{equation}\label{R_1_Cform}
 \begin{array}{rclrcl}
\widetilde{R}_{11}&=&\displaystyle\frac{\nu}{1+\nu}I + 
\frac{4\delta_1\nu}{(1+\nu)^2}K_{\kappa_2} \quad&
\widetilde{R}_{12}&=&\displaystyle -\frac{2}{1+\nu} S_{\kappa_1} 
-\frac{4\delta_2(1-\nu)}{(1+\nu)^2}S_{\kappa_2} K_{\kappa_2}^\top  \\
\widetilde{R}_{21}&=&\displaystyle\frac{2\nu}{1+\nu}N_{\kappa_1}+\frac{4\delta_2\nu(1-\nu)}{
(1+\nu)^2}N_{\kappa_2} K_{\kappa_2} \quad&
\widetilde{R}_{22}&=&\displaystyle\frac{1}{1+\nu}I-\frac{4\delta_1\nu}{(1+\nu)^2}K_{\kappa_2}
^\top
\end{array}
\end{equation}
for wavenumbers $\kappa_j$ such that $\Re\kappa_j\geq 0$ and $\Im \kappa_j>0$ 
for $j=1,2$. In equations~\eqref{R_1_Cform} $\delta_1=\delta_2=0$ in the case of the 
two-dimensional operators $\widetilde{\mathcal{R}}_1^{s,s}$ defined in 
equations~\eqref{eq:tildR1} and in the case of the three-dimensional operators 
$\widetilde{\mathcal{R}}_1^{s,s-1}$ defined in equations~\eqref{eq:tildR13dm1}, $\delta_1=\delta_2=1$ in the case of the three-dimensional operators 
$\widetilde{\mathcal{R}}_1^{s,s}$ defined in equations~\eqref{eq:tildR13d}, and $\delta_1=1,\ \delta_2=0$ in the case of the three-dimensional operators 
$\widetilde{\mathcal{R}}_1^{1,s,s}$ defined in equations~\eqref{eq:tildR13dN}. We 
have

\begin{theorem}\label{thm_well_posedness}
 Let $\widetilde{\mathcal{R}}_1$ be defined as in equations~\eqref{R_1_Cform} 
and let
  $\kappa_1$ be a wavenumber such that $\Re\kappa_1>0$ and $\Im\kappa_1>0$, and 
$\kappa_2=i\epsilon,\ \epsilon>0$. Then
\begin{enumerate}
\item in the case $d=2$ the operators $\widetilde{\mathcal{D}}^{s,s}$ defined 
in Theorem~\ref{thm2} are invertible with continuous inverses in the spaces 
$H^{s}(\Gamma)\times H^{s}(\Gamma)$ for all $s$;
\item in the case $d=3$ the operators $\widetilde{\mathcal{D}}^{s,s-1}$ defined 
in Theorem~\ref{thm23dm1} are invertible with continuous inverses in the spaces 
$H^{s}(\Gamma)\times H^{s-1}(\Gamma)$ for all $s$;
\item in the case $d=3$ the operators $\widetilde{\mathcal{D}}^{s,s-1}$ defined 
in Theorem~\ref{thm23dm1} are invertible with continuous inverses in the spaces 
$H^{s}(\Gamma)\times H^{s}(\Gamma)$ for all $s$;
\item in the case $d=3$ the operators $\widetilde{\mathcal{D}}^{s,s}$ defined 
in Theorem~\ref{thm23d} are invertible with continuous inverses in the spaces 
$H^{s}(\Gamma)\times H^{s}(\Gamma)$ for all $s$.
\item in the case $d=3$ the operators $\widetilde{\mathcal{D}}^{1,s,s}$ defined 
in Theorem~\ref{thm23dN} are invertible with continuous inverses in the spaces 
$H^{s}(\Gamma)\times H^{s}(\Gamma)$ for all $s$.
\end{enumerate}
\end{theorem}
\begin{proof}
Given that the operators $\widetilde{\mathcal{D}}^{s,s}$ are Fredholm of the 
second kind in the spaces $H^{s}(\Gamma)\times H^{s}(\Gamma)$ for both $d=2$ 
and $d=3$ (see Theorem~\ref{thm2} and Theorem~\ref{thm23d} respectively), the operators $\widetilde{\mathcal{D}}^{1,s,s}$ are Fredholm of the 
second kind in the spaces $H^{s}(\Gamma)\times H^{s}(\Gamma)$ for $d=3$ (see Theorem~\ref{thm23dN}), and 
the operators $\widetilde{\mathcal{D}}^{s,s-1}$ are Fredholm of the second kind 
in the spaces $H^{s}(\Gamma)\times H^{s-1}(\Gamma)$ for $d=3$ (see 
Theorem~\ref{thm23dm1}) and in the spaces $H^{s}(\Gamma)\times H^{s}(\Gamma)$ for $d=3$ (see Theorem~\ref{thm23dm1ss}), it suffices to establish that these operators are 
injective in the corresponding spaces in order to conclude their invertibility 
and the continuity of their inverses. Let us assume that $(a,b)$ is in the 
kernel of a generic operator $\widetilde{\mathcal{D}}$ corresponding to the 
generic regularizing operator $\widetilde{R}_{ij},\ i,j=1,2$ defined in 
equations~\eqref{R_1_Cform}. Let us define
\begin{equation}
u^1(\mathbf{z})= [DL_1(\widetilde{R}_{11} a+\widetilde{R}_{12} 
b)](\mathbf{z})-[SL_1(\widetilde{R}_{21} a+\widetilde{R}_{22} 
b)](\mathbf{z}),\quad \mathbf{z}\in\mathbb{R}^d\setminus\Gamma\nonumber
\end{equation}
and
\begin{equation}
u^2(\mathbf{z})= -[DL_2(\widetilde{R}_{11} a+\widetilde{R}_{12} 
b-a)](\mathbf{z})+\nu^{-1} [SL_2(\widetilde{R}_{21} a+\widetilde{R}_{22} 
b-b)](\mathbf{z}),\quad \mathbf{z}\in\mathbb{R}^d\setminus\Gamma.\nonumber
\end{equation}
Obviously $u^1|_{D_1}$ and $u^2|_{D_2}$ are solutions of the transmission 
problem~\eqref{eq:Ac_i}-\eqref{eq:bc}. Given that the wavenumbers $k_1$ and $k_2$ are 
real, classical results about uniqueness of transmission 
problems~\cite{KressRoach} give us that $u_1=0$ in $D_1$ and $u_2=0$ in $D_2$ 
which implies that $\gamma_D^1u^1=\gamma_N^1u^1=0$ and 
$\gamma_D^2u^2=\gamma_N^2u^2=0$. We use the well known jump formulas of the 
layer potentials and we get
\begin{eqnarray}\label{jump_final}
\gamma_D^2u^1=-\widetilde{R}_{11}a-\widetilde{R}_{12} 
b,&&\gamma_N^2u^1=-\widetilde{R}_{21} a-\widetilde{R}_{22} b\nonumber\\
\gamma_D^1u^2=a-\widetilde{R}_{11} a-\widetilde{R}_{12} 
b,&&\gamma_N^1u^2=\nu^{-1}(b-\widetilde{R}_{21}a-\widetilde{R}_{22} b).
\end{eqnarray}
We have then
\begin{eqnarray}
\nu \int_\Gamma \gamma_D^1u^2\ \overline{\gamma_N^1u^2}d\sigma&=&\int_\Gamma a\ 
\overline{b}\ d\sigma -\int_\Gamma a\ \overline{\widetilde{R}_{21}a}\ {d\sigma} 
- \int_\Gamma a\ \overline{\widetilde{R}_{22}b}\ {d\sigma}\nonumber\\
&&-\int_\Gamma (\widetilde{R}_{11}a)\ \overline{b}\ {d\sigma} - \int_\Gamma 
(\widetilde{R}_{12}b)\ \overline{b}\ {d\sigma} + \int_\Gamma \gamma_D^2u^1 
\overline{\gamma_N^2u^1}{d\sigma}.\nonumber
\end{eqnarray}
The previous relation is equivalent to 
\begin{eqnarray}
\nu \int_\Gamma \gamma_D^1u^2\ 
\overline{\gamma_N^1u^2}{d\sigma}&=&-\frac{2\nu}{1+\nu}\int_\Gamma a\ 
\overline{N_{\kappa_1} a}\ {d\sigma} 
+\frac{4\nu\delta_1}{(1+\nu)^2}\left(\int_\Gamma a\ \overline{K_{i\epsilon}^\top 
b}\ {d\sigma} -\int_\Gamma K_{i\epsilon}a\ \overline{b}\ 
{d\sigma}\right)\nonumber\\
&&+\frac{2}{1+\nu}\int_\Gamma (S_{\kappa_1} b)\ \overline{b}\ 
{d\sigma}+\frac{4\delta_2(1-\nu)}{(1+\nu)^2}\int_\Gamma 
(S_{i\epsilon}K_{i\epsilon}^\top b)\ \overline{b}\ {d\sigma}\nonumber\\
&&-\frac{4\delta_2\nu(1-\nu)}{(1+\nu)^2}\int_\Gamma a\ 
\overline{N_{i\epsilon}K_{i\epsilon}a}\ {d\sigma} +\int_\Gamma \gamma_D^2u^1 
\overline{\gamma_N^2u^1}{d\sigma}\nonumber
\end{eqnarray}
which can be also written as
\begin{eqnarray}\label{eq:uniq1}
\nu \int_\Gamma \gamma_D^1u^2\ 
\overline{\gamma_N^1u^2}{d\sigma}&=&-\frac{2\nu}{1+\nu}
\int_\Gamma 
a\ \overline{N_{\kappa_1}a}\ {d\sigma}+\frac{2}{1+\nu}\int_\Gamma (S_{\kappa_1} b)\ 
\overline{b}\ {d\sigma}\nonumber\\
&&+\frac{4\delta_2(1-\nu)}{(1+\nu)^2}\int_\Gamma(S_{i\epsilon}K_{i\epsilon}^\top 
b)\ \overline{b}\ {d\sigma}-\frac{4\delta_2\nu(1-\nu)}{(1+\nu)^2}\int_\Gamma a\ 
\overline{N_{i\epsilon}K_{i\epsilon}a}\ {d\sigma} \nonumber\\
&&+\int_{D_2}(-k_1^2|u|^2+|\nabla u|^2)dx
\end{eqnarray}
if we use Green's identities and the fact that $K_{i\epsilon}$ is the adjoint 
of the operator $K_{i\epsilon}^\top$ with respect to the complex scalar product 
on $L^2(\Gamma)$. If we take the imaginary part  in  both sides of 
equation~\eqref{eq:uniq1} and use the result from Lemma \ref{lemma:A01}, namely
$$\Im \int_\Gamma (S_{i\epsilon}K_{i\epsilon}^\top b)\ \overline{b}\ 
{d\sigma}=0\qquad \Im \int_\Gamma a\ \overline{N_{i\epsilon}K_{i\epsilon}a}\ 
{d\sigma}=0$$ 
we conclude
\begin{equation}
\nu \Im \int_\Gamma \gamma_D^1u^2\ \overline{\gamma_N^1u^2}{d\sigma}=\frac{2\nu}{1+\nu}\ \Im \int_\Gamma 
(N_{\kappa_1} a)\ \overline{a}\ {d\sigma} + \frac{2}{1+\nu}\Im \int_\Gamma 
(S_{\kappa_1} b)\ \overline{b}\ {d\sigma}.\nonumber\\
\end{equation} 
But we have the following positivity property~\cite{turc1} 
\[
\Im \int_\Gamma (S_{\kappa_1} \varphi)\ \overline{\varphi}\ {d\sigma}>0,\ 
\varphi\neq 0,\quad 
\Im \int_\Gamma (N_{\kappa_1} \psi)\ \overline{\psi}\ {d\sigma}>0,\ \psi\neq 0
\]
and thus
\[
\Im \int_\Gamma \gamma_D^1u^2\ \overline{\gamma_N^1u^2}{d\sigma}\geq 0.
\]
Since $u^2$ is a radiative solution of the Helmholtz equation with
wavenumber $k_2$ in the domain $D_1$ it
follows~\cite[p. 78]{KressColton} that $u^2=0$ in $D_1$, and,
thus
\[
\Im \int_\Gamma (S_{\kappa_1} a)\ \overline{a}\ {d\sigma}=0,\qquad \Im \int_\Gamma (N_{\kappa_1} b)\ \overline{b}\ {d\sigma}=0.
\]
Therefore, necessary   both $a=0$  and $b=0$ on $\Gamma$ which proves the 
theorem.
\end{proof}

\begin{remark} As pointed out in~\cite{turc2}, another possible choice of regularizing operators that leads to results qualitatively similar to those in Theorem~\ref{thm_well_posedness} consists of Fourier multipliers whose symbols are equal to the principal symbols of the boundary layer operators featured in equations~\eqref{R_1_Cform} when the latter are viewed as pseudodifferential operators. In the three dimensions, the principal symbols of the latter operators can be expressed in terms of the variable $\mathbf{\xi}\in TM^{*}(\Gamma)$ (which represents the Fourier symbol of the tangential gradient operator $\nabla_\Gamma$) where $TM^{*}(\Gamma)$ represents the cotangent bundle of $\Gamma$~\cite{Taylor}:
\begin{equation}\label{R_1_PS}
 \begin{array}{rclrcl}
\sigma(N_{\kappa_1})(\mathbf{x},\xi)&=&\displaystyle -\frac{1}{2}(|\xi|^2-\kappa_1^2)^{1/2} \quad&
\sigma(S_{\kappa_1})(\mathbf{x},\xi)&=&\displaystyle \frac{1}{2(|\xi|^2-\kappa_1^2)^{1/2}}  \\
\sigma(K_{\kappa_2})(\mathbf{x},\xi)&=&\displaystyle \frac{\mathcal{K}(\mathbf{x})\xi\cdot\xi}{2(|\xi|^2-\kappa_2^2)^{3/2}}-\frac{H(\mathbf{x})}{(|\xi|^2-\kappa_2^2)^{1/2}}\quad&
\sigma(K_{\kappa_2}^\top)(\mathbf{x},\xi)&=&\displaystyle \sigma(K_{\kappa_2})(\mathbf{x},\xi).
\end{array}
\end{equation}
In equations~\eqref{R_1_PS} above, $\mathcal{K}(\mathbf{x})=\nabla{\mathbf{n}}(\mathbf{x})$ is the curvature tensor of the surface $\Gamma$ at $\mathbf{x}\in\Gamma$ and $H(\mathbf{x})$ is the principal curvature of the surface $\Gamma$ at $\mathbf{x}\in\Gamma$. We recall that a Fourier multiplier $\mathcal{S}$ with symbol $\sigma(\mathbf{x},\xi)$ is defined as
\[
(\mathcal{S}\varphi)(\mathbf{x})=\int \sigma(\mathbf{x},\xi)e^{i\mathbf{x}\cdot \xi}\widehat{\varphi}(\xi)d\xi
\]
where $\varphi$ as a function defined on $\Gamma$ and $\hat{\varphi}$ is its Fourier transform on the manifold $\Gamma$~\cite{Taylor}. In two dimensions the principal symbols $\sigma(N_{\kappa_1})(\mathbf{x},\xi)$ and $\sigma(S_{\kappa_1})(\mathbf{x},\xi)$ are the same as in three dimensions, while $\sigma(K_{\kappa_1})(\mathbf{x},\xi)=\sigma(K_{\kappa_1}^\top)(\mathbf{x},\xi)=0$.
\end{remark}

\section{Conclusions}

We presented regularized Combined Field Integral
Equations formulations for the solution of acoustic transmission problems. In 
this context, the regularizing operator is defined naturally as an 
approximation to the admittance operator that maps the boundary data of 
transmission problems (that is differences of Dirichlet and Neumann data on the 
interface of material discontinuity) to the Cauchy data on the boundary of each 
medium. The construction of the regularizing operators relies on approximations
of Dirichlet-to-Neumann operators in each medium via suitable boundary layer 
operators with complex wavenumbers and Calder\'on's calculus. Certain 
positivity properties of the imaginary parts of boundary layer operators with 
complex wavenumbers are used in order to prove the well-posedness of the 
formulations. As shown elsewhere~\cite{turc2}, solvers based on the new 
formulations outperform solvers based on other existing integral formulations 
of transmission problems. The extension of this work to the electromagnetic 
case is currently underway. Finally, we mention that in the case $d=2$, we can select the complex 
wavenumber $\kappa$ in the definition~\eqref{eq:tildR1} of the regularizing 
operator $\widetilde{\mathcal{R}}_1^{s,s}$ so that solvers based on the 
formulation GCSIE with corresponding integral operators 
$\widetilde{\mathcal{D}}^{s,s}$ outperform solvers based on the classical 
integral formulations of transmission 
problems~\cite{KressRoach,KittapaKleinman,rokhlin-dielectric} in the 
high-contrast, high-frequency regime~\cite{turc2}. This is also the case in 
three dimensions as confirmed by our preliminary results~\cite{turc3}.

\section*{Acknowledgments}
 Yassine Boubendir gratefully acknowledge support from NSF through contract 
DMS-1319720. Catalin Turc gratefully acknowledge support from NSF through 
contract DMS-1312169. V\'{\i}ctor Dom\'{\i}nguez is partially 
supported by MICINN
Project MTM2010-21037.  Part of this research was carried out
during a short visit of Prof. V\'{\i}ctor Dom\'{\i}nguez to NJIT.

\appendix
\section{Positiveness properties for some operators}\label{app1}

We prove in this section the following result: 
\begin{lemma}\label{lemma:A01} Let $\epsilon>0$. Then, for any $(a,b)\in 
H^{1/2}(\Gamma)\times H^{-1/2}(\Gamma)$ it holds
$$\Im \int_\Gamma (S_{i\epsilon}K_{i\epsilon}^\top b)\ \overline{b}\ 
{d\sigma}=0,\qquad \Im \int_\Gamma a\ \overline{N_{i\epsilon}K_{i\epsilon}a}\ 
{d\sigma}=0.$$
\end{lemma}
\begin{proof} Observe that $S_{i\varepsilon} K^\top_{i\varepsilon}$ is a real self-adjoint operator. Indeed, if $\varphi$ and $psi$ are real functions defined on $\Gamma$, then  taking into account Calder\'on' s identity $S_{i\varepsilon} K^\top_{i\varepsilon}=K_{i\varepsilon} S_{i\varepsilon}$~\eqref{eq:calderon} we obtain 
\[
\int_{\Gamma} \varphi\ S_{i\varepsilon} K^\top_{i\varepsilon} \psi d\sigma=
\int_{\Gamma} \varphi\ K_{i\varepsilon} S_{i\varepsilon} \psi\ d\sigma =\int_\Gamma K_{i\varepsilon}^\top \varphi\ S_{i\varepsilon}\psi\ d\sigma=\int_{\Gamma} \psi\ S_{i\varepsilon} 
K_{i\varepsilon}^\top \varphi\ d\sigma.
\]
 Then, given $b\in H^{-1/2}(\Gamma)$ such that $b=\varphi+i\psi$ where $\varphi$ and $psi$ are real functions defined on $\Gamma$, we have 
\[
 \Im \int_{\Gamma} S_{i\varepsilon} K^\top_{i\varepsilon} b\ \overline{b}\ d\sigma=\int_\Gamma \varphi\ S_{i\varepsilon} K^\top_{i\varepsilon} \psi\ d\sigma - \int_\Gamma \psi\ S_{i\varepsilon} K^\top_{i\varepsilon} \varphi\ d\sigma = 0.
\]
The proof for $N_{i\varepsilon}K_{i\varepsilon} $ is analogous.
\end{proof}

\bibliographystyle{hplain}
\bibliography{biblioTransmission.bib}

\end{document}